\documentclass[11pt,a4paper,reqno]{amsart}
\fontsize{6.5pt}{6.5pt}\selectfont 
\makeatletter
\def\@evenhead{{\fontsize{6.5pt}{6.5pt}\selectfont \hfil \leftmark\hfil\thepage}}
\def\@oddhead{{\fontsize{6.5pt}{6.5pt}\selectfont\hfil\rightmark \hfil\thepage}}
\makeatother
\usepackage{amsfonts}
\usepackage{booktabs}
\usepackage{amsmath,amssymb,amsthm,amsxtra}
\usepackage{float}
\usepackage{amssymb}
\usepackage{mathabx}
\usepackage{bbm}
\usepackage[colorlinks, linkcolor=blue,anchorcolor=Periwinkle,
citecolor=Red,urlcolor=Emrald]{hyperref}
\usepackage[usenames,dvipsnames]{xcolor}
\usepackage{enumitem}
\usepackage{geometry,array} 
\usepackage{graphicx}
\usepackage{subfigure}
\usepackage{bookmark}
\usepackage{tikz}\usetikzlibrary{matrix,intersections, calc, arrows.meta}
\usepackage{url}
\usepackage{dsfont}
\usepackage[colorinlistoftodos]{todonotes}
\usepackage{tablists} \restorelistitem
\usepackage{bm}

\usetikzlibrary{decorations.markings}
\tikzset{->-/.style={decoration={  markings,  mark=at position #1 with
    {\arrow{>}}},postaction={decorate}}}
\tikzset{-<-/.style={decoration={  markings,  mark=at position #1 with
    {\arrow{<}}},postaction={decorate}}}
\usepackage{extarrows}
\usepackage[all]{xy}
\usepackage{setspace}
\usepackage{thmtools}
\usepackage{thm-restate}
\usepackage{hyperref}
\usepackage{cleveref}
\usepackage{multirow}
\usepackage{ifthen}
\usepackage{tikz-3dplot}
\usepackage{comment}
\usepackage{mathrsfs}

\newcommand{\diag}{\operatorname{diag}}

\newcommand{\id}{\operatorname{id}}

\newcommand{\sign}{\operatorname{sign}}

\newcommand{\mfC}{\mathbf{C}}

\newcommand{\mfE}{\mathbf{E}}

\newcommand{\mfG}{\mathbf{G}}

\newcommand{\mfS}{\mathbf{S}}

\newcommand{\mfc}{\mathbf{c}}

\newcommand{\mfe}{\mathbf{e}}

\newcommand{\mfu}{\mathbf{u}}

\newcommand{\mfw}{\mathbf{w}}
\newcommand{\mfx}{\mathbf{x}}

\newcommand{\mcA}{\mathcal{A}}

\newcommand{\mcD}{\mathcal{D}}

\newcommand{\mcM}{\mathcal{M}}

\newcommand{\mcO}{\mathcal{O}}
\newcommand{\mcP}{\mathcal{P}}

\newcommand{\mcT}{\mathcal{T}}

\newcommand{\mbN}{\mathbb{N}}

\newcommand{\mbR}{\mathbb{R}}

\newcommand{\mbT}{\mathbb{T}}

\theoremstyle{plain}
\newtheorem{theorem}{Theorem}[section]

\newtheorem{lemma}[theorem]{Lemma}
\newtheorem{corollary}[theorem]{Corollary}
\newtheorem{proposition}[theorem]{Proposition}
\newtheorem{conjecture}[theorem]{Conjecture}
\theoremstyle{definition}
\newtheorem{definition}[theorem]{Definition}

\newtheorem{example}[theorem]{Example}

\newtheorem{remark}[theorem]{Remark}

\newtheorem{question}[theorem]{Question}

\numberwithin{equation}{section}
\newtheorem{definition-proposition}[theorem]{Definition-Proposition}

\begin{document}

\title {Sign-coherence and tropical sign pattern for rank 3 real cluster-cyclic exchange matrices}

\date{\today}
\author{Ryota Akagi}
\address{Graduate School of Mathematics\\ nagoya University\\Chikusa-ku\\ Nagoya\\464-0813\\ Japan.}
\email{ryota.akagi.e6@math.nagoya-u.ac.jp}

\author{Zhichao Chen}
\address{School of Mathematical Sciences\\ University of Science and Technology of China \\ Hefei, Anhui 230026, P. R. China.}
\email{czc98@mail.ustc.edu.cn}

\maketitle
\begin{abstract}
Sign-coherence about $c$-vectors was conjectured by Fomin-Zelevinsky and proved by Gross-Hacking-Keel-Kontsevich for the integer skew-symmetrizable case. We prove this conjecture associated with $c$-vectors for the rank 3 real cluster-cyclic skew-symmetrizable case. Simultaneously, we establish their self-contained recursion and monotonicity. Then, these $c$-vectors are proved to be roots of certain quadratic equations. Based on these results, we prove that the corresponding exchange graphs of $C$-patterns and $G$-patterns are $3$-regular trees.  We also study the structure of tropical signs and equip the dihedral group $\mcD_6$ with a cluster realization via certain mutations. 
\end{abstract}
\tableofcontents
\section{Introduction}\label{intro}
Cluster algebras are a class of commutative algebras distinguished by their dynamic sets of generators and relations. They were originally introduced in \cite{FZ02, FZ03} as a tool to study total positivity in Lie groups and canonical bases in quantum groups. The main object is the \emph{seed}, which consists of cluster variables, coefficients and an integer skew-symmetrizable matrix. Its transformation is called a \emph{mutation}.  By gathering some special information of seeds, the integer $C$-matrices and $G$-matrices can be defined \cite{FZ07}. They play an important role in the study of cluster algebras \cite{DWZ10, Pla11, NZ12, BHIT17, GHKK18, CGY22, LMN24}.

In \cite{FZ07, NZ12}, the recursions (\Cref{defCG}) for $C$-, $G$-matrices were introduced. By considering this recursion, we can generalize their definition for any real entries. In this paper, we focus on $C$-, $G$-matrices with real entries.

A fundamental problem in (ordinary) cluster algebras is \emph{sign-coherence} of $C$-patterns, which was given by \cite{FZ07}. It was proved in the skew-symmetric case by \cite{DWZ10, Pla11, Nag13} with the method of algebraic representation theory. In general, that is for the skew-symmetrizable case, this conjecture was proved by \cite{GHKK18} with the method of scattering diagrams. Note that the exchange matrices, $C$-matrices and $G$-matrices of (ordinary) cluster algebras are integer matrices. On the other hand, we focus on the real ones. Hence, the following natural question arises. 
\begin{question}
	Is the $C$-pattern corresponding to a real exchange matrix sign-coherent?
\end{question}
In general, the answer to this question is no. In \cite{AC25}, we classified all real rank $2$ cases and finite type cases whose $C$-patterns ($G$-patterns) are sign-coherent. In particular, finite type can be classified by the Coxeter diagrams. These results suggest that the extension to real exchange matrices retains meaningful combinatorial structures and is therefore worthy of further investigation.

In this paper, we aim to deal with the much more complicated and infinite real case of rank $3$. There is a special class of exchange matrices called \emph{cluster-cyclic exchange matrices} and the others are said to be \emph{cluster-acyclic}, see \Cref{cc and ca}. In particular, sufficient and necessary conditions for a real exchange matrix to be cluster-cyclic are given by \cite{Sev12, Aka24}. We prove the sign-coherence for the real cluster-cyclic case of rank $3$ as follows.  
\begin{theorem}[{\Cref{thm: sign-coherency}}] \label{intro main1}
	Every $C$-pattern corresponding to a real cluster-cyclic exchange matrix is sign-coherent.
\end{theorem}
Let $\mcT$ be the set of all reduced sequences $\mfw$, which are used to record the path of cluster mutations. We also define the tropical sign pattern $\varepsilon(B)$ (\Cref{var-pattern}) associated with the real exchange matrix $B$. To prove the main theorem above, we need to simultaneously prove the monotonicity of $c$-vectors (\Cref{lem: inequalities}) and recursion formulas of the corresponding tropical sign pattern as follows.
\begin{theorem}[{\Cref{thm: recursion for tropical signs}}] \label{intro main2}
	Let $B \in \mathrm{M}_3(\mathbb{R})$ be a cluster-cyclic exchange matrix. Let ${\bf w} \in \mathcal{T}\backslash\{\emptyset\}$ and $k$ be the last index of ${\bf w}$.\\
$(a)$\ There is a unique $s \in \{1,2,3\}\backslash\{k\}$ such that
\begin{equation}\label{intro: eq: choice of s}
\varepsilon_{s}^{\bf w}\mathrm{sign}(b_{ks}^{\bf w})=-1,\ \varepsilon_{k}^{\bf w} \neq \varepsilon_{s}^{\bf w}.
\end{equation}
$(b)$\ Let $s$ be the index defined by $(a)$ and $t \in \{1,2,3\}\backslash\{k,s\}$ be the other index. Then, the following equalities hold:
\begin{equation}\label{intro: eq: recursion for tropical signs}
\begin{aligned}
(\varepsilon_k^{{\bf w}[s]},\varepsilon_s^{{\bf w}[s]},\varepsilon_t^{{\bf w}[s]}) &= (-\varepsilon_k^{{\bf w}},-\varepsilon_s^{{\bf w}},\varepsilon_t^{{\bf w}}),\\
(\varepsilon_k^{{\bf w}[t]},\varepsilon_s^{{\bf w}[t]},\varepsilon_t^{{\bf w}[t]}) &= (\varepsilon_k^{{\bf w}},\varepsilon_s^{{\bf w}},-\varepsilon_t^{{\bf w}}).
\end{aligned}
\end{equation}
\end{theorem}
It is worth mentioning that there are some counter-examples for the cluster-acyclic exchange matrices, see \Cref{counter}. When $B$ is an integer skew-symmetric matrix (quiver type) in acyclic cluster algebras of rank $n$, positive $c$-vectors are real Schur roots, namely the dimension vectors of indecomposable rigid modules \cite{ST13, Cha15, HK16}. Then, more relations among them  are studied by \cite{Sev15, FT18, Ngu22, LLM23, EJLN24}. However, the properties of the $c$-vectors associated with cluster-cyclic exchange matrices are still mysterious. For the rank 3 real cluster-cyclic case, we define the \emph{symmetrized quasi-Cartan companion} by $\tilde{A}_{k_1}$ and $\tilde{A}^{\mfw}$ in \Cref{quasi-Cartan companion of P}. Then, based on \Cref{intro main1} and \Cref{intro main2}, we give a quasi-Cartan congruence relation of $C$-matrices as follows. 
\begin{theorem}[{\Cref{duality}}]
	Let $B \in \mathrm{M}_3(\mathbb{R})$ be a cluster-cyclic exchange matrix with the skew-symmetrizer $D$ and $\mfw=[k_1,\dots,k_r]$ be any reduced sequence with $r\geq 1$. Then, the following quasi-Cartan congruence relation holds:
	\begin{align}
		(C^{\mfw})^{\top}\tilde{A}_{k_1}C^{\mfw}=\tilde{A}^{\mfw}. \label{intro: qCc}
	\end{align}
\end{theorem}
In particular, if we focus on the diagonal entries of \eqref{intro: qCc}, we may obtain that the $c$-vectors are roots of some quadratic equations, see \Cref{main equation}. As an application, we can prove \Cref{conjecture: for real entries} for this case (\Cref{prop: conjectures are true}), which was first proposed in \cite{AC25}. Furthermore, we study two important exchange graphs of the $C$-pattern $\mfE \mfG(\mfC)$ and the $G$-pattern $\mfE \mfG(\mfG)$ defined by \Cref{CG-EG}. In the following, we exhibit their corresponding $3$-regular tree structure.
\begin{theorem}[{\Cref{thm: exchange graph C}}]
For any cluster-cyclic exchange matrix $B \in \mathrm{M}_3(\mathbb{R})$, the exchange graph of the $C$-pattern $\mfE \mfG(\mfC)$ and the exchange graph of the $G$-pattern $\mfE \mfG(\mfG)$ are $3$-regular trees.
\end{theorem}
As an important corollary of this theorem, we obtain the following combinatorial property.
\begin{corollary}[\Cref{cor: exchange graph}]
For the  cluster-cyclic skew-symmetrizable cluster algebra $\mcA$ of rank $3$, the following statements hold:
\begin{enumerate}
	\item[$(1)$] The exchange graph of the cluster pattern $\mfE \mfG(\bf{\Sigma})$ is a $3$-regular tree.
	\item[$(2)$] The exchange graph of the $C$-pattern $\mfE \mfG(\mfC)$ is a $3$-regular tree.
	\item[$(3)$] The exchange graph of the $G$-pattern $\mfE \mfG(\mfG)$ is a $3$-regular tree.
\end{enumerate}
	
\end{corollary}
Furthermore, we investigate more important properties of tropical signs, whose recursion is given by Theorem~\ref{intro main2}. For this purpose, we separate the set $\mcT$ into two types of subsets, called the \emph{trunk} and the \emph{branch}, see \Cref{definition: tree of sequences}. Since a trunk consists only of special mutations, its structure is relatively simple, see Proposition~\ref{tropical trunk}. In contrast, the structure of a branch is more complicated. Based on Theorem~\ref{intro main2}, we slightly change the notation of mutations to a monoid action of a certain monoid $\mathcal{M}$ generated by formal symbols $S$ and $T$, see Definition~\ref{def: monoid action of M}. We then introduce its quotient monoid $\overline{\mathcal{M}}$ as \eqref{eq: equivalent relation of M}. Here, $\bar{S}$ and $\bar{T}$ denote the equivalence classes in $\overline{\mathcal{M}}$. Then, the group structure of $\overline{\mathcal{M}}$, which reflects a mutation of tropical signs in a branch, is related to the dihedral group structure $\mcD_6$.
\begin{theorem}[{\Cref{thm: group structure of M}}]
The quotient monoid $\overline{\mathcal{M}}$ is a group which is isomorphic to the dihedral group $\mathcal{D}_{6}$ of order $12$. In particular, the following relations hold:
\begin{equation}
\bar{S}^2=\bar{T}^6=(\bar{S}\bar{T})^2=\mathrm{id}.
\end{equation}
\end{theorem}
Based on this theorem, we can find a fractal structure appearing in the $\varepsilon$-pattern, see Theorem~\ref{thm: fractal structure in E-pattern}. Moreover, we can associate the mutation of tropical signs with an explicit transformation in $\mathbb{R}^2$, see Section~\ref{geometric characterization}.

This paper is organized as follows. In \Cref{pre}, we review some basic notions and properties about ordinary cluster algebras. In particular, we generalize the notions of integer exchange matrices, $C$-matrices and $G$-matrices to real ones, see \Cref{defCG}. To study the column sign-coherence of the real $C$-matrices, we introduce the tropical sign pattern, see \Cref{var-pattern}. Then, two exchange graphs of the $C,G$-pattern are provided by \Cref{CG-EG}. In \Cref{Sec: sign-coh}, we focus on the tropical sign pattern for rank $3$ real cluster-cyclic exchange matrices. Based on this pattern, we show the column sign-coherence of $C$-matrices (\Cref{thm: sign-coherency}) and the monotonicity of $c$-vectors (\Cref{lem: inequalities}). In \Cref{Sec: quadratic}, we prove that the $c$-vectors associated with the cluster-cyclic exchange matrices are the solutions to some quadratic equations, see \Cref{main equation}. In \Cref{sec: proof of the exchange graph}, we show the $3$-regular tree structure of the exchange graphs of $C,G$-patterns associated with  rank $3$ real cluster-cyclic exchange matrices, see \Cref{thm: exchange graph C}. In \Cref{sec: Structure of tropical signs}, we exhibit the fractal structure of tropical signs for rank $3$ real cluster-cyclic exchange matrices (\Cref{tropical trunk}, \Cref{tropical branch} $\&$ \Cref{thm: fractal structure in E-pattern}) and the group structure of certain mutations (\Cref{thm: group structure of M}). In \Cref{geometric characterization}, we give a geometric model of tropical signs (\Cref{geo model graph}) and a cluster realization of the dihedral group $\mcD_6$ (\Cref{cluster real}).
\section*{Conventions}
We define some special matrices and notations as follows.
\begin{itemize}[leftmargin=2em]\itemsep=0pt
	\item Let $I \in \mathrm{M}_n(\mathbb{R})$ be the identity matrix. For any $k=1,2,\dots,n$, define a diagonal matrix $J_{k}=\mathrm{diag}(a_1,a_2,\dots,a_n)$ as follows: $a_i=1$ if $i\neq k$ and $a_k=-1$. 
	\item For any matrix $A=(a_{ij})_{m\times n} \in \mathrm{M}_{m \times n}(\mathbb{R})$ and $k \in \mathbb{Z}_{\geq 1}$, define $A^{\bullet k}=(b_{ij})_{m\times n}$ as follows: $b_{ik} = a_{ik}$ for any $i=1,\dots,m$ and $b_{ij}=0$ for $j\neq k$. Similarly, define $A^{k \bullet}=(c_{ij})_{m\times n}$ as $c_{kj}=a_{kj}$ for any $j=1,\dots,n$ and $c_{ij}=0$ for $i \neq k$. 
	\item For a real number $a \in \mathbb{R}$, define $[a]_{+}=\max(a,0)$. For any matrix $A=(a_{ij})_{m\times n}$, define $[A]_{+}=([a_{ij}]_{+})_{m\times n}$.
	\item For a finite set $\mfS$, we denote the number of elements in $\mfS$ by $\#\mfS$.
\end{itemize}
\section{Preliminaries}\label{pre}

\subsection{Real exchange matrices and $C,G$-matrices}\

In this subsection,  we generalize and define real exchange matrices and $C,G$-matrices, which are analogous to the integer ones. Beforehand, we introduce notation for the mutation sequence.

Consider a sequence ${\bf w}=[k_1,\dots,k_r]$, where $k_i\in \{1,2,\dots,n\}$. A sequence ${\bf w}$ is said to be {\em reduced} if $k_i \neq k_{i+1}$ for any $i=1,2,\dots,r-1$. By convention, the empty sequence $\emptyset = [\ ]$ is also reduced. We write the set of all reduced sequences by $\mathcal{T}$. For any reduced sequence ${\bf w}=[k_1,\cdots,k_r] \in \mathcal{T}$, the {\em length} of $\bf w$ is defined by $|\mfw|=r$.
\par
For any $k\in\{1,\dots,n\}$, define ${\bf w}[k] \in \mathcal{T}$ as ${\bf w}[k]=[k_1,\dots,k_r,k]$ if $k_r\neq k$, and ${\bf w}[k]=[k_1,\dots,k_{r-1}]$ if $k=k_r$.  
For any two reduced sequences ${\bf w},{\bf u}=[l_1,\dots,l_s] \in \mathcal{T}$, we define the product as ${\bf w}{\bf u}={\bf w}[l_1][l_2]\cdots[l_s]$.
\begin{definition}[\emph{Real exchange matrix, $C,G$-matrices}] \label{defCG}
Let $B \in \mathrm{M}_n(\mathbb{R})$ be a \emph{real} skew-symmetrizable matrix.
For any $k\in\{1,2,\dots,n\}$, the {\em mutation} of $B$ in direction $k$ is defined by
\begin{align}
\mu_k(B)=(J_k+[-B]^{\bullet k}_{+})B(J_k+[B]^{k \bullet}_{+}).\label{exchange mut}
\end{align}
For any ${\bf w}=[k_1,\dots,k_r] \in \mathcal{T}$, we denote by $B^{\bf w}=\mu_{k_r}\cdots\mu_{k_1}(B)$ and call it a \emph{real exchange matrix}.
We define real {\em $C$-matrices} $C^{\bf w}$ and {\em $G$-matrices} $G^{\bf w}$ by the following recursion:
\begin{equation}\label{eq: recursion of C}
\begin{aligned}
C^{\emptyset}&=G^{\emptyset}=I,\\
C^{{\bf w}[k]}&=C^{\bf w}J_k+C^{\bf w}[B^{\bf w}]^{k \bullet}_{+} +[-C^{\bf w}]^{\bullet k}_{+}B^{\bf w},\\
G^{{\bf w}[k]}&=G^{\bf w}J_k+G^{\bf w}[-B^{\bf w}]^{\bullet k}_{+}-B^{\emptyset}[-C^{\bf w}]^{\bullet k}_{+}.
\end{aligned}
\end{equation}
We denote the mutations of $C^{\mfw}$ and $G^{\mfw}$ at $k$ respectively by $
\mu_k(C^{\mfw})=C^{\mfw[k]}$ and 
$\mu_k(G^{\mfw})=G^{\mfw[k]}.$
 The collections of all $C$-matrices and $G$-matrices are called the {\em $C$-pattern} and the {\em $G$-pattern}, and we denote them by ${\bf C}(B)=\{C^{\bf w}\}_{{\bf w} \in \mathcal{T}}$ and ${\bf G}(B)=\{G^{\bf w}\}_{{\bf w} \in \mathcal{T}}$ respectively.
Each column (resp. row) vector of $C^{\bf w}$ (resp. $G^{\bf w}$) is called a \emph{$c$-vector} (resp. \emph{$g$-vector}). We write $C^{\mathbf{w}}=(\mathbf{c}_1^{\mathbf{w}},\dots,\mathbf{c}_n^{\mathbf{w}})$ and $G^{\mathbf{w}}=(\mathbf{g}_1^{\mathbf{w}},\dots,\mathbf{g}_n^{\mathbf{w}})$ to denote these vectors.
\end{definition}
We can easily check that the mutation of $B$ is an involution, that is, $\mu_k(\mu_k(B))=B$ holds for any skew-symmetrizable matrix $B$. We say that two skew-symmetrizable matrices $B,B'$ are {\em mutation-equivalent} if there exists ${\bf w} \in \mathcal{T}$ such that $B'=B^{\bf w}$. Since the mutation of $B$ is an involution, mutation-equivalence is an equivalence relation on the set of all skew-symmetrizable matrices.
\begin{remark}
The mutation formula \eqref{exchange mut} of exchange matrices and the recursion formula \eqref{eq: recursion of C} of integer $C,G$-matrices come from the classical cluster algebras, see \cite{NZ12} and \cite[Eq.(1.28) \& Eq.(1.37)]{Nak23}. It naturally motivates us to generalize and define the same recursion formulas for real ones.
\end{remark}
Note that real $C,G$-matrices also have the following \emph{first duality} relation.
\begin{proposition}[{\cite[Prop.~II.1.17]{Nak23}}]
Let $B \in \mathrm{M}_n(\mathbb{R})$ be skew-symmetrizable. Then, for any ${\bf w} \in \mathcal{T}$, the first duality holds:
\begin{equation}\label{eq: first duality}
G^{\bf w}B^{\bf w}=B^{\emptyset}C^{\bf w}.
\end{equation}
\end{proposition}

We introduce a partial order $\leq$ on $\mathbb{R}^n$ as follows: for any $\mathbf{u}=(u_1,\dots,u_n)^{\top}, \mathbf{v} = (v_1,\dots,v_n)^{\top} \in \mathbb{R}^n$, $\mathbf{u} \leq \mathbf{v}$ holds if and only if $u_i \leq v_i$ for any $i = 1,\dots,n$. Moreover, $\mathbf{u} < \mathbf{v}$ means $\mathbf{u} \leq \mathbf{v}$ and $\mathbf{u} \neq \mathbf{v}$. Since column sign-coherence \cite{GHKK18} is a key notion for classical $C$-matrices, it is natural to generalize this notion to the real case. 
\begin{definition}[\emph{$\varepsilon$-pattern}] \label{var-pattern}
Let $B \in \mathrm{M}_n(\mathbb{R})$ be a real exchange matrix. Then, the $C$-matrix $C^{\bf w}$ is said to be {\em sign-coherent} if either ${\bf c}^{\bf w}_{i} > {\bf 0}$ or ${\bf c}^{\bf w}_i < {\bf 0}$ holds for each $i=1,2,\dots,n$. The $C$-pattern ${\bf C}(B)$ is called {\em sign-coherent} if every $C^{\bf w}$ is sign-coherent. When $C^{\bf w}$ is sign-coherent, the {\em tropical sign} $\varepsilon_{i}^{\bf w}$ of ${\bf c}^{\bf w}_{i}$ is defined as follows: $\varepsilon_{i}^{\bf w} \in \{\pm 1\}$ such that $\varepsilon_{i}^{\bf w}{\bf c}_{i}^{\bf w} > {\bf 0}$. When ${\bf C}(B)$ is sign-coherent, we define the {\em $\varepsilon$-pattern} $\bm{\varepsilon}(B)$ as a collection of $n$-tuples of signs $\{(\varepsilon_1^{\bf w},\varepsilon_2^{\bf w},\dots,\varepsilon_n^{\bf w})\}_{{\bf w} \in \mathcal{T}}$.
\end{definition}
Here, we define the column sign-coherence of real $C$-matrices and we can similarly define the row sign-coherence of real $G$-matrices.
\begin{proposition}\label{unimodularity}
Let $B \in \mathrm{M}_n(\mathbb{R})$ be a skew-symmetrizable matrix.\\
\textup{$(a)$} For any $\varepsilon=\pm 1$, we have the following relations.
\begin{equation}\label{eq: epsilon expression of C, G}
\begin{aligned}
C^{{\bf w}[k]}&=C^{\bf w}J_k+C^{\bf w}[\varepsilon B^{\bf w}]^{k \bullet}_{+}+[-\varepsilon C^{\bf w}]^{\bullet k}_{+}B^{\bf w},\\
G^{{\bf w}[k]}&=G^{\bf w}J_k+G^{\bf w}[-\varepsilon B^{\bf w}]^{\bullet k}_{+}-B^{\emptyset}[-\varepsilon C^{\bf w}]^{\bullet k}_{+}. 
\end{aligned}
\end{equation}
In particular, if $C^{\bf w}$ is sign-coherent, then the recursion \eqref{eq: recursion of C} can be expressed as
\begin{equation}
\begin{aligned}\label{CG rec}
C^{{\bf w}[k]}&=C^{\bf w}(J_k+[\varepsilon_k^{\bf w}B^{\bf w}]^{k \bullet}_{+}),\\
G^{{\bf w}[k]}&=G^{\bf w}(J_k+[-\varepsilon_{k}^{\bf w} B^{\bf w}]^{\bullet k}_{+}).
\end{aligned}
\end{equation}
\textup{$(b)$} If $C^{\mathbf{w}}$ is sign-coherent, we obtain the recursion for $c$-vectors and $g$-vectors as
\begin{equation}\label{eq: mutation of c,g-vectors}
\begin{aligned}
{\bf c}^{{\bf w}[k]}_{i}&=\begin{cases}
-{\bf c}^{\bf w}_{k} & i=k,\\
{\bf c}^{\bf w}_{i}+[\varepsilon_{k}^{\bf w}b^{\bf w}_{ki}]_{+}{\bf c}^{\bf w}_{k} & i \neq k,
\end{cases}\\
{\bf g}^{{\bf w}[k]}_i&=\begin{cases}
-{\bf g}^{\bf w}_{k} + \sum_{j=1}^{n}[-\varepsilon_k^{\bf w}b_{jk}^{\bf w}]_{+}{\bf g}^{\bf w}_{j} & i=k,\\
{\bf g}^{\bf w}_{i} & i \neq k.
\end{cases}
\end{aligned}
\end{equation}
\end{proposition}
\begin{proof}
	A similar method as \cite[Eq.(2.4)]{NZ12} can be applied to \eqref{eq: epsilon expression of C, G}. Then, it directly implies $(b)$. \end{proof}
A non-zero $c$-vector is called \emph{green} (resp. \emph{red}) if all its components are nonnegative (resp. nonpositive). 
\begin{definition}[\emph{$m$-reddening sequence}]
Let $B\in \mathrm{M}_n(\mathbb{R})$ be a skew-symmetrizable matrix with ${\bf C}(B)$ sign-coherent. Then, the tropical sign of each $c$-vector is $1$ or $-1$. Let $\mfw=[k_1,\dots,k_r]$ be a reduced mutation sequence.  For any $1\leq i\leq r$, we denote by $\mfw_{i}=[k_1,\dots,k_i]$ and $\mfw_{0}=\emptyset$. We say that $\bf w$ is a \emph{reddening sequence} of $B$ if all the tropical signs of $C^{\bf w}$ are $-1$. The reddening sequence $\bf w$ is an \emph{$m$-reddening sequence} if $\#\{\varepsilon^{{\bf w}_{i}}_{k_{i+1}}=-1|\, 0\leq i\leq r-1\}=m$. In particular, the $0$-reddening sequence is called a \emph{maximal green sequence}. 
\end{definition}
\begin{remark}\label{rmk sign coherent}
In \cite{GHKK18}, it was shown that every ${C}$-pattern corresponding to an {\em integer}  skew-symmetrizable matrix is sign-coherent. However, it is still an open problem for {\em real} skew-symmetrizable matrices.
\end{remark}
For an integer $d \in \mathbb{Z}_{\geq 0}$, we say that \emph{sign-coherence holds up to $d$} if $C^{\mathbf{w}}$ is sign-coherent whenever $|\mathbf{w}| \leq d$, where $|\mathbf{w}|$ denotes the length of the sequence $\mathbf{w}$.
\begin{proposition}[{\cite[Prop.~4.2, Eq.(3.11)]{NZ12}}, {\cite[Prop.~4.8]{AC25}}]\label{prop: fundamental properties under sign-coherency}
Consider a $C$-pattern $\{C^{\bf w}\}_{\mfw\in \mcT}$ corresponding to a {\em real} skew-symmetrizable matrix $B=(b_{ij})$. Suppose that sign-coherence holds up to $d \in \mathbb{Z}_{\geq 0}$. Then, for any $\mathbf{w} \in \mathcal{T}$ with $|\mathbf{w}| \leq d+1$, we have the following claims.
\\
\textup{$(a)$} We have $|C^{\bf w}|=|G^{\bf w}|=(-1)^{\pm |{\bf w}|}$. In particular, $\{{\bf c}_{i}^{\bf w} \mid i=1,\dots,n\}$ and $\{{\bf g}_{i}^{\bf w} \mid i=1,\dots,n\}$ each forms a basis of the vector space $\mathbb{R}^n$.
\\
\textup{$(b)$} The second duality relation holds:
\begin{equation}\label{eq: second duality}
D^{-1}(C^{\bf w})^{\top}DG^{\bf w}=I.
\end{equation}
\end{proposition}
The proof can be referred to \cite[Eq.(3.11)]{NZ12} and \cite[Prop 2.3]{Nak23} similarly.

\subsection{Exchange graphs of $C,G$-patterns}\

In the study of ordinary cluster algebras, we sometimes focus on their combinatorial structures. Two fundamental combinatorial objects associated with an ordinary cluster algebra are the cluster complex and the exchange graph, which are constructed from cluster variables and seeds, respectively, in the sense of \cite{FZ02,FZ03}. Here, we generalize the notion of an exchange graph by using $C$, $G$-patterns. Let $\mathfrak{S}_n$ be the symmetric group of degree $n$.
\begin{definition}[\emph{Two permutation actions}]
	Let $\sigma \in \mathfrak{S}_n$ and $A=(a_{ij})_{n\times n} \in \mathrm{M}_{n}(\mbR)$. We define two \emph{permutation actions} $\sigma(A), \tilde{\sigma}(A)$ as follows:
	\begin{enumerate}
		\item $\sigma(A) =(a_{\sigma^{-1}(i)\sigma^{-1}(j)})_{n\times n}$.
		\item  $\tilde{\sigma}(A)=(a_{i\sigma^{-1}(j)})_{n\times n}.$
	\end{enumerate}
\end{definition}
Now, we introduce the permutation matrix $P_\sigma$ associated with $\sigma\in \mathfrak{S}_n$ as follows:
\begin{align}
	P_{\sigma}=\tilde{\sigma}(I).
\end{align} Then, we can get the following lemma by direct calculation.
\begin{lemma}
	For any $A,B \in \mathrm{M}_{n}(\mathbb{R})$ and $\sigma \in \mathfrak{S}_n$, the following equalities hold:
	\begin{enumerate}
		\item[$(1)$] $\sigma(A)=P_{\sigma}^{\top}AP_{\sigma},\ \tilde{\sigma}(A)=AP_{\sigma}$.
		\item[$(2)$] $\sigma(AB)=\sigma(A)\sigma(B),\ \tilde{\sigma}(AB)=A\tilde{\sigma}(B).$
	\end{enumerate}
\end{lemma}
The following proposition can be generalized directly from the integer $C,G$-matrices to real ones, which is called the \emph{compatibility relation}.
\begin{proposition}[{\cite[Eq.(3.17)]{Nak21}}]
	For any two triples $(B^{\mfw},C^{\mfw},G^{\mfw})$, $(\sigma(B^{\mfw}),\tilde{\sigma}(C^{\mfw}),\tilde{\sigma}(G^{\mfw}))$, and a permutation $\sigma\in \mathfrak{S}_n$, the compatibility relation holds:
	\begin{enumerate}
		\item[$(1)$] $\mu_{\sigma(k)}(\tilde{\sigma}(C^{\mfw}))=\tilde{\sigma}(\mu_k(C^{\mfw}))$.
		\item[$(2)$] $\mu_{\sigma(k)}(\tilde{\sigma}(G^{\mfw}))=\tilde{\sigma}(\mu_k(G^{\mfw}))$.
	\end{enumerate}
\end{proposition}

For any two real $C$-matrices $C^{\mfw}$ and $C^{\mfu}$ in the same $C$-pattern, we define an equivalence relation: $C^{\mfw}\sim C^{\mfu}$ if there exists a permutation $\sigma\in \mathfrak{S}_n$ such that $C^{\mfu}=\tilde{\sigma}(C^{\mfw})$. Then, each equivalence class $[C^{\mfw}]$ is called an \emph{unlabeled $C$-matrix}. Similarly, in a $G$-pattern, we can define $[G^{\mfw}]$ to be an \emph{unlabeled $G$-matrix}.
\par
To introduce the exchange graph, we define the {\em quotient graph}.
Let $G=(V,E)$ be a graph with a vertex set $V$ and an edge set $E \subset V\times V$. Let $\sim$ be an equivalence relation on $V$. Then, we define the {\em quotient graph} $\tilde{G}=G/{\sim}$ as follows:
\begin{itemize}
\item The vertex set of $\tilde{G}$ is the equivalence classes of $V/{\sim}$.
\item Two vertices $[v_1],[v_2] \in \tilde{G}$ are connected in $G/{\sim}$ if and only if there exist vertices $v'_1 \in [v_1]$ and $v'_2 \in [v_2]$ such that $v'_1$ and $v'_2$ are connected in $G$.
\end{itemize}
We may naturally identify $\mathcal{T}$ as the $n$-regular tree by the following rule: ${\bf w},{\bf u} \in \mathcal{T}$ are connected if and only if there exists $k=1,2,\dots,n$ such that ${\bf u}={\bf w}[k]$. (This is graph isomorphic to the $n$-regular tree $\mbT_n$.) Then, the exchange graph can be defined as follows.
\begin{definition}[\emph{Exchange graph of $C,G$-patterns}] \label{CG-EG} 
    The \emph{exchange graph} $\mfE \mfG(\mfC)$ (resp. $\mfE \mfG(\mfG)$) of a $C$-pattern (resp. $G$-pattern) is defined as the quotient graph $\mathcal{T}/{\sim}$, where for any ${\bf w},{\bf u} \in \mathcal{T}$,
    \begin{equation}
        {\bf w} \sim {\bf u} \Longleftrightarrow [C^{\bf w}]=[C^{\bf u}] \quad (\textup{resp}.\ [G^{\bf w}]=[G^{\bf u}]).
    \end{equation}
    We often replace each vertex $\overline{\bf w}$ of ${\bf EG}({\bf C})$ and ${\bf EG}({\bf G})$ by the corresponding unlabeled $C$-matrix $[C^{\bf w}]$ and the unlabeled $G$-matrix $[G^{\bf w}]$, respectively. By the above definition, this replacement is independent of the choice of ${\bf w}$.
\end{definition}

Lastly, we recall the relationship between the ordinary exchange graph of cluster algebras in \cite{FZ02} and these exchange graphs. In ordinary (integer) cluster algebras, an exchange graph is defined by unlabeled {\em seeds} $\Sigma^{\bf w}=({\bf x}^{\bf w},B^{\bf w})$, where ${\bf x}^{\bf w}=(x_1^{\bf w},\dots,x_n^{\bf w})$ is a tuple of {\em cluster variables}. Note that in \cite{FZ02}, they are indexed by the vertices of an $n$-regular tree $\mathbb{T}_n$. By fixing one initial vertex $t_0 \in \mathbb{T}_{n}$, there is a natural one-to-one correspondence between $t\in \mbT_n$ and ${\bf w}\in \mcT$, and we write ${\bf x}_t={\bf x}^{\bf w}$. We define $\sigma{\Sigma}^{\bf w}=(\sigma{\bf x}^{\bf w},\sigma B^{\bf w})$, where $\sigma{\bf x}^{\bf w}=(x_{\sigma^{-1}(1)}^{\bf w},x_{\sigma^{-1}(2)}^{\bf w},\dots,x_{\sigma^{-1}(n)}^{\bf w})$. As the following theorem indicates, the periodicity of seeds coincides with that of $C$, $G$-matrices.
\begin{theorem}[{\cite[Synchronicity]{Nak21}}]\label{Synchronicity Nak}
	Let $\bf{\Sigma}$ be any cluster pattern with an {\em integer} initial exchange matrix $B \in \mathrm{M}_{n}(\mathbb{Z})$. Then, for any $\mfw,\mfu \in \mcT$ and $\sigma\in \mathfrak{S}_n$, the following are equivalent:
	\begin{enumerate}
		\item[$(1)$] $\Sigma^{\mfw}=\sigma(\Sigma^{\mfu}).$
		\item[$(2)$] $C^{\mfw}=\tilde{\sigma}(C^{\mfu}).$
		\item[$(3)$] $G^{\mfw}=\tilde{\sigma}(G^{\mfu}).$
	\end{enumerate}
In particular, the ordinary exchange graph is graph isomorphic to that of $C$, $G$-patterns.
\end{theorem}
As the above theorem indicates, in  ordinary (integer) cluster algebras, we do not have to care about the difference between $C$-matrices and $G$-matrices if we focus on their periodicity. However, by generalizing them to real cases, we do not know whether this is true in general. In \cite{AC25}, we showed it under the following assumptions.
\begin{conjecture}[{\cite[Conjecture 5.10]{AC25}}]\label{conjecture: for real entries}
For a given real skew-symmetrizable matrix $B$ with a skew-symmetrizer $D=\mathrm{diag}(d_1,\dots,d_n)$, suppose that ${\bf C}(B)$ is sign-coherent. Then, the following statements hold.
\\
\textup{($a$)} For any $B'$ which is mutation-equivalent to $B$, two $C$-patterns ${\bf C}(B')$ and ${\bf C}((B')^{\top})$ are sign-coherent.
\\
\textup{($b$)} For any $B'$ which is mutation-equivalent to $B$, consider its $C$-pattern ${\bf C}(B')=\{C^{\bf w}_{B'}\}_{{\bf w} \in \mathcal{T}}$. If its $c$-vector ${\bf c}_{i;B'}^{\bf w}$ is expressed as $\alpha {\bf e}_{j}$ for some $\alpha \in \mathbb{R}$ and $j=1,2,\dots,n$, we have $\alpha=\pm\sqrt{d_id_j^{-1}}$.
\end{conjecture}
\begin{proposition}[{\cite[Theorem 10.7]{AC25}}]\label{prop: two exchange graphs are the same}
Let $B$ be a real skew-symmetrizable matrix. Suppose that  Conjecture~\ref{conjecture: for real entries} holds. Then, the two exchange graphs of the $C$-pattern and the $G$-pattern are graph isomorphic.
\end{proposition}
\begin{remark}\label{rem: bad example of g-vector fan}
In ordinary cluster algebras, another combinatorial structure known as the {\em cluster complex} is also studied. This can be realized as a {\em $g$-vector fan}. Here, the $g$-vector fan is the set of all $g$-vector cones and their faces, where each $g$-vector cone is spanned by the $g$-vectors in a $G$-matrix \cite{Rea14, RS16, AC25}. Under Conjecture~\ref{conjecture: for real entries}, it is known that each $g$-vector fan is indeed a fan consisting of $n$-dimensional maximal cones \cite{AC25}. We focus on the \emph{dual graph} of the $g$-vector fan: the vertex set consists of the maximal cones, and two maximal cones $\sigma, \tau$ are connected by an edge if the codimension of $\sigma \cap \tau$ is $1$. For the integer case, it is known that this dual graph coincides with the above exchange graphs. However, by considering real entries, they might be different. Here, we provide an example of rank $2$ \cite[Example~6.7]{AC25}. Consider the $G$-pattern with the initial exchange matrix $B= \left(\begin{smallmatrix}
    0 & -\frac{1}{2}\\
    2 & 0
\end{smallmatrix}\right)$. Then, by a direct calculation, we verify that the corresponding exchange graph consists of the blue edges in \Cref{fig: bad example}, which is a $10$-cycle. On the other hand, as illustrated in \Cref{fig: bad example}, the corresponding $g$-vector fan consists of $5$ maximal cones, and hence the dual graph is a $5$-cycle.
\begin{figure}[hbtp]
\begin{tikzpicture}
\draw (-3,0)--(3,0);
\draw (0,-2.5)--(0,3);
\draw (0,0)--(1.25,-2.5);
\draw[red] (1,1) -- (-1,1) -- (-1,-1) -- (3/8,-3/2)--(1.5,-0.75) -- (1,1);
\fill[red] (1,1) circle [radius=0.06];
\fill[red] (-1,1) circle [radius=0.06];
\fill[red] (-1,-1) circle [radius=0.06];
\fill[red] (3/8,-3/2) circle [radius=0.06];
\fill[red] (1.5,-0.75) circle [radius=0.06];

\draw (2,2) node {$\left(\begin{smallmatrix}
1 & 0\\
0 & 1
\end{smallmatrix}\right)$};
\draw (-2,2) node {$\left(\begin{smallmatrix}
-1 & 0\\
0 & 1
\end{smallmatrix}\right)$};
\draw (-2,-2) node {$\left(\begin{smallmatrix}
-1 & 0\\
0 & -1
\end{smallmatrix}\right)$};
\draw (3/4,-3) node {$\left(\begin{smallmatrix}
1 & 0\\
-2 & -1
\end{smallmatrix}\right)$};
\draw (3,-1.5) node {$\left(\begin{smallmatrix}
1 & \frac{1}{2}\\
-2 & 0
\end{smallmatrix}\right)$};
\draw (4,3) node {$\left(\begin{smallmatrix}
0 & \frac{1}{2}\\
2 & 0
\end{smallmatrix}\right)$};
\draw (-3,3) node {$\left(\begin{smallmatrix}
0 & -\frac{1}{2}\\
2 & 0
\end{smallmatrix}\right)$};
\draw (-3,-3) node {$\left(\begin{smallmatrix}
0 & -\frac{1}{2}\\
-2 & 0
\end{smallmatrix}\right)$};
\draw (9/8,-4) node {$\left(\begin{smallmatrix}
0 & \frac{1}{2}\\
-2 & -1
\end{smallmatrix}\right)$};
\draw (5.5,-9/4) node {$\left(\begin{smallmatrix}
1 & \frac{1}{2}\\
0 & -1
\end{smallmatrix}\right)$};
\draw[blue] (1.2,2)--(-1.1,2);
\draw[blue] (-2,1.5)--(-2,-1.5);
\draw[blue] (-1.6,-2.3)--(0,-3);
\draw[blue] (1.6,-2.8)--(2.8,-1.9);
\draw[blue] (3.1,-1)--(4,2.5);
\draw[blue] (3,3)--(-2.1,3);
\draw[blue] (-3.2,2.5)--(-3.2,-2.5);
\draw[blue] (-2.3,-3.5)--(0.2,-4);
\draw[blue] (2,-3.8)--(4.8,-2.8);
\draw[blue] (5,-1.8)--(2.2,1.7);
\end{tikzpicture}
\caption{$G$-pattern and $g$-vector fan associated with $B=\left(\begin{smallmatrix}
0 & -\frac{1}{2}\\
2 & 0
\end{smallmatrix}\right)$.}
\label{fig: bad example}
\end{figure}
\end{remark}

\section{Tropical signs corresponding to cluster-cyclic exchange matrices} \label{Sec: sign-coh}
For any matrix (or vector) $A=(a_{ij})$, we define $\mathrm{sign}(A)$ as the same size matrix whose entries are the signs of corresponding entries of $A$. Here, we focus on the following $3\times 3$ real matrices.
\begin{definition} \label{cc and ca}
A skew-symmetrizable matrix $B \in \mathrm{M}_3(\mathbb{R})$ is called {\em cyclic} if
\begin{equation}
\mathrm{sign}(B)=\pm\left(\begin{matrix}
0 & -1 & 1\\
1 & 0 & -1\\
-1 & 1 & 0
\end{matrix}\right).
\end{equation}
If every $B^{\bf w}$ (${\bf w} \in \mathcal{T}$) is cyclic, $B$ is said to be {\em cluster-cyclic}.
\end{definition}
If $B$ is cyclic and skew-symmetric, then it corresponds to a quiver with a directed $3$-cycle.
\par
We may check that, if $B$ is cluster-cyclic, then $\mathrm{sign}(B^{\bf w})=(-1)^{|{\bf w}|}\mathrm{sign}(B)$. The following fact is known.
\begin{proposition}[{\cite{BBH11, Sev12, Aka24}}]\label{prop: lowerbound of b}
Let $B=(b_{ij}) \in \mathrm{M}_3(\mathbb{R})$ be a cyclic exchange matrix. Then, the following two conditions are equivalent:
\begin{itemize}
\item $B$ is cluster-cyclic.
\item $|b_{ij}b_{ji}| \geq 4$ for any $i,j \in \{1,2,3\}$ with $i\neq j$ and
\begin{equation}
|b_{12}b_{21}|+|b_{23}b_{32}|+|b_{31}b_{13}|-|b_{12}b_{23}b_{31}| \leq 4.
\end{equation}
\end{itemize}
\end{proposition}
\subsection{Sign-coherence and the recursion for tropical signs}\

In this subsection, we show the sign-coherence for rank $3$ real cluster-cyclic exchange matrices.
\begin{theorem}\label{thm: sign-coherency}
Every $C$-pattern corresponding to a {\em real} cluster-cyclic exchange matrix is sign-coherent.
\end{theorem}
Moreover, tropical signs may be obtained by the following recursion.
\begin{theorem}\label{thm: recursion for tropical signs}
Let $B \in \mathrm{M}_3(\mathbb{R})$ be a cluster-cyclic exchange matrix. Let ${\bf w} \in \mathcal{T}\backslash\{\emptyset\}$ and $k$ be the last index of ${\bf w}$.\\
$(a)$\ There is a unique $s \in \{1,2,3\}\backslash\{k\}$ such that
\begin{equation}\label{eq: choice of s}
\varepsilon_{s}^{\bf w}\mathrm{sign}(b_{ks}^{\bf w})=-1,\ \varepsilon_{k}^{\bf w} \neq \varepsilon_{s}^{\bf w}.
\end{equation}
$(b)$\ Let $s$ be the index defined by $(a)$ and $t \in \{1,2,3\}\backslash\{k,s\}$ be the other index. Then, the following equalities hold:
\begin{equation}\label{eq: recursion for tropical signs}
\begin{aligned}
(\varepsilon_k^{{\bf w}[s]},\varepsilon_s^{{\bf w}[s]},\varepsilon_t^{{\bf w}[s]}) &= (-\varepsilon_k^{{\bf w}},-\varepsilon_s^{{\bf w}},\varepsilon_t^{{\bf w}}),\\
(\varepsilon_k^{{\bf w}[t]},\varepsilon_s^{{\bf w}[t]},\varepsilon_t^{{\bf w}[t]}) &= (\varepsilon_k^{{\bf w}},\varepsilon_s^{{\bf w}},-\varepsilon_t^{{\bf w}}).
\end{aligned}
\end{equation}
\end{theorem}
Namely, if we focus on a rank $3$ real cluster-cyclic exchange matrix, its tropical signs may be controlled by certain rules, although this is difficult in general for higher rank. This simplified behavior of tropical signs is a key reason why we can extend the sign-coherent property for any rank $3$ {\em real} cluster-cyclic exchange matrix.
\par
By considering this recursion, we have the following property.
\begin{corollary}\label{sign-recursion}
Let $B \in \mathrm{M}_3(\mathbb{R})$ be a cluster-cyclic exchange matrix. Then, the $\varepsilon$-pattern corresponding to $B$ may be obtained recursively by (\ref{eq: recursion for tropical signs}) and the following initial conditions:
\begin{equation}\label{eq: initial conditions of tropical signs}
\varepsilon_{i}^{\emptyset}=1,\ \varepsilon_{i}^{[k]}=\begin{cases}
-1 & i=k,\\
1 & i \neq k.
\end{cases}
\end{equation}
Moreover, the $\varepsilon$-pattern corresponding to a cluster-cyclic exchange matrix $B$ is determined by $\mathrm{sign}(B)$. 
\end{corollary}
Moreover, by \Cref{thm: recursion for tropical signs}~$(a)$, we get the following corollary.
\begin{corollary}\label{two signs}
Consider the $\varepsilon$-pattern corresponding to a real cluster-cyclic exchange matrix $B$. Then, there is no ${\bf w} \in \mathcal{T}\backslash\{\emptyset\}$ such that
\begin{equation}
(\varepsilon_{1}^{\bf w},\varepsilon_{2}^{\bf w},\varepsilon_3^{\bf w})=(1,1,1)\ \text{or}\ (-1,-1,-1). 
\end{equation} In particular, $B$ admits neither a reddening sequence nor a maximal green sequence.
\end{corollary}
\begin{remark}\label{MGS}
	In \cite[Theorem 1.2 \& Theorem 1.4]{Sev14}, Seven proved that an integer cluster-cyclic exchange matrix admits neither a reddening sequence nor a maximal green sequence by use of mutation rules. Here, our method is more general for any real cluster-cyclic case and we also exclude the existence of the tropical sign $(1,1,1)$ except the initial one.
\end{remark}
\begin{example}
Based on \Cref{sign-recursion}, for tropical signs of rank $3$ real cluster-cyclic exchange matrices, we can refer to \Cref{fig: example of tropical signs} as an example \footnote{For each matrix in Figure~\ref{fig: example of tropical signs}, the top $3\times 3$ block is $\mathrm{sign}(B^{\bf w})$ and the bottom row is $(\varepsilon_{1}^{\bf w},\varepsilon_2^{\bf w},\varepsilon_3^{\bf w})$.}
.
\begin{figure}[h]
\begin{tikzpicture}
\draw (0,0) node
{$\underset{\textup{initial matrix}}{\left(\begin{matrix}
0 & - & +\\
+ & 0 & -\\
- & + & 0\\\hline
+ & + & +
\end{matrix}\right)}$};
\draw (0,-3.5) node
{$\underset{k=1,\ s=3}{\left(\begin{matrix}
0 & + & -\\
- & 0 & +\\
+ & - & 0\\\hline
- & + & +
\end{matrix}\right)}$};
\draw (-3.5,-6.5) node
{$\underset{k=2,\ s=3}{\left(\begin{matrix}
0 & - & +\\
+ & 0 & -\\
- & + & 0\\\hline
- & - & +
\end{matrix}\right)}$};
\draw (3.5,-6.5) node
{$\underset{k=3,\ s=1}{\left(\begin{matrix}
0 & - & +\\
+ & 0 & -\\
- & + & 0\\\hline
+ & + & -
\end{matrix}\right)}$};
\draw (-5,-10) node
{$\underset{k=1,\ s=2}{\left(\begin{matrix}
0 & + & -\\
- & 0 & +\\
+ & - & 0\\\hline
+ & - & +
\end{matrix}\right)}$};
\draw (-2,-10) node
{$\underset{k=3,\ s=2}{\left(\begin{matrix}
0 & + & -\\
- & 0 & +\\
+ & - & 0\\\hline
- & + & -
\end{matrix}\right)}$};
\draw (2,-10) node
{$\underset{k=1,\ s=3}{\left(\begin{matrix}
0 & + & -\\
- & 0 & +\\
+ & - & 0\\\hline
- & + & +
\end{matrix}\right)}$};
\draw (5,-10) node
{$\underset{k=2,\ s=1}{\left(\begin{matrix}
0 & + & -\\
- & 0 & +\\
+ & - & 0\\\hline
+ & - & -
\end{matrix}\right)}$};
\draw[->] (0,-1.5)--(0,-1.75) node [right] {$1$}--(0,-2);
\draw[->] (-1.5,-4)--(-2.25,-4.5) node [above left] {$t=2$}--(-3,-5);
\draw[->] (1.5,-4)--(2.25,-4.5) node [above right] {$s=3$}--(3,-5);
\draw[->] (-4,-7.7)--(-4.5,-8.2) node [above left] {$t=1$}--(-5,-8.7);
\draw[->] (-3,-7.7)--(-2.5,-8.2) node [above right] {$s=3$}--(-2,-8.7);
\draw[->] (3,-7.7)--(2.5,-8.2) node [above left] {$s=1$}--(2,-8.7);
\draw[->] (4,-7.7)--(4.5,-8.2) node [above right] {$t=2$}--(5,-8.7);
\end{tikzpicture}
\caption{Tropical signs for real cluster-cyclic matrices}\label{fig: example of tropical signs}
\end{figure}
\end{example}
\subsection{Main proof and the monotonicity of $c$-vectors}\

Now, let us prove Theorem~\ref{thm: sign-coherency} and Theorem~\ref{thm: recursion for tropical signs}. For the proof, we need to show the following inequalities simultaneously.
\begin{lemma}\label{lem: inequalities}
Let $B \in \mathrm{M}_3(\mathbb{R})$ be a real cluster-cyclic exchange matrix. Let ${\bf w} \in \mathcal{T}\backslash\{\emptyset\}$. Set $k,s,t \in \{1,2,3\}$ as in Theorem~\ref{thm: recursion for tropical signs}~$(b)$.
\\
$(a)$ We have the following inequality:
\begin{equation}\label{eq: k,s inequality}
\varepsilon_s^{\bf w}((|b^{\bf w}_{sk}b_{ks}^{\bf w}| - 2){\bf c}_{k}^{\bf w}+|b^{\bf w}_{sk}|{\bf c}_{s}^{\bf w}) \geq {\bf 0}.
\end{equation}
$(b)$\ For any $k' \in \{1,2,3\} \backslash\{k\}$ and $i=1,2,3$, we have
\begin{equation}\label{eq: monotonicity}
\varepsilon_{i}^{\bf w}{\bf c}_i^{\bf w} \leq \varepsilon_{i}^{{\bf w}[k']}{\bf c}_{i}^{{\bf w}[k']}.
\end{equation}
Moreover, there exists $i\in \{1,2,3\}$ such that the equality does not hold, that is, the inequality is strict with at least one entry.
\end{lemma}

\begin{proof}[Proof of Theorem~\ref{thm: sign-coherency}, Theorem~\ref{thm: recursion for tropical signs}, and Lemma~\ref{lem: inequalities}]
For simplicity, we assume
\begin{equation}
B=\left(\begin{matrix}
0 & -w & y\\
z & 0 & -u\\
-v & x & 0
\end{matrix}\right),
\end{equation}
for some $x,y,z,u,v,w>0$ and set
\begin{equation}
B^{\bf w}=\left(\begin{matrix}
0 & -(-1)^{|{\bf w}|}w^{\bf w} & (-1)^{|{\bf w}|}y^{\bf w}\\
(-1)^{|{\bf w}|}z^{\bf w} & 0 & -(-1)^{|{\bf w}|}u^{\bf w}\\
-(-1)^{|{\bf w}|}v^{\bf w} & (-1)^{|{\bf w}|}x^{\bf w} & 0
\end{matrix}\right)
\end{equation}
with $x^{\bf w},y^{\bf w},z^{\bf w},u^{\bf w},v^{\bf w},w^{\bf w} > 0$.
By Proposition~\ref{prop: lowerbound of b}, we have $x^{\bf w}u^{\bf w},y^{\bf w}v^{\bf w},z^{\bf w}w^{\bf w} \geq 4$. 
For any $r \in \mathbb{Z}_{\geq 0}$, let $\mathcal{T}^{\leq r}=\{{\bf w} \in \mathcal{T}\mid |{\bf w}| \leq r\}$.
Define the statements
\begin{itemize}
\item[$(1)_r$] For any ${\bf w} \in \mathcal{T}^{\leq r}$, $C^{\bf w}$ is sign-coherent.
\item[$(2)_r$] For any ${\bf w} \in \mathcal{T}^{\leq r}$, Theorem~\ref{thm: recursion for tropical signs}~(a) holds.
\item[$(3)_r$] For any ${\bf w} \in \mathcal{T}^{\leq r}$, Lemma~\ref{lem: inequalities}~(a) holds.
\item[$(4)_r$] For any ${\bf w} \in \mathcal{T}^{\leq r}$, Theorem~\ref{thm: recursion for tropical signs}~(b) and Lemma~\ref{lem: inequalities}~(b) hold.
\end{itemize}

We show $(1)_{r}$, $(2)_r$, $(3)_{r}$, and $(4)_{r-1}$ by induction on $r\geq 2$. When $r=2$, all claims hold by a direct calculation. (See Figure~\ref{fig: r=1}.) When we prove $(3)_2$, we can use $yv\geq 4$ by Proposition~\ref{prop: lowerbound of b}.
\begin{figure}[htbp]
\centering
\begin{tikzpicture}[scale=0.8]
\node (O) at (-4.2,0) {\tiny ${\left(\begin{array}{ccc}
0 & -w & y\\
z & 0 & -u\\
-v & x & 0\\ \hline
1 & 0 & 0\\
0 & 1 & 0\\
0 & 0 & 1
\end{array}\right)}$};
\node at (O.south) [below] {\small initial matrix};
\node (A) at (0,0) {\tiny ${
\left(\begin{array}{ccc}
0 & w & -y\\
-z & 0 & u^{[1]}\\
v & -x^{[1]} & 0\\ \hline
-1 & 0 & y\\
0 & 1 & 0\\
0 & 0 & 1
\end{array}\right)}$};
\node at (A.south) [below] {\small $k=1$, $s=3$};
\draw (6,1.5) node {\tiny $\left(\begin{array}{ccc}
0 & -w & y^{[1,2]}\\
z & 0 & -u^{[1]}\\
-v^{[1,2]} & x^{[1]} & 0\\ \hline
-1 & 0 & y\\
0 & -1 & u^{[1]}\\
0 & 0 & 1
\end{array}\right)$};
\draw (6,-1.5) node {\tiny $\left(\begin{array}{ccc}
0 & -w^{[1,3]} & y\\
z^{[1,3]} & 0 & -u^{[1]}\\
-v & x^{[1]} & 0\\ \hline
yv-1 & 0 & -y\\
0 & 1 & 0\\
v & 0 & -1
\end{array}\right)$};
\draw[->] (-2.5,0) -- (-2.2,0) node [above]{$1$} -- (-1.9,0);
\draw[->] (2,0.2) -- (2.7,0.6) node [above=3pt]{\tiny $t=2$} -- (3.4,1.0);
\draw[->] (2,-0.2) -- (2.7,-0.6) node [below=3pt]{\tiny $s=3$} -- (3.4,-1.0);
\end{tikzpicture}
\caption{The base case for $r=2$} \label{fig: r=1}
\end{figure}

Suppose that $(1)_{r}$, $(2)_{r}$, $(3)_r$ and $(4)_{r-1}$ hold for some $r \in \mathbb{Z}_{\geq 2}$. Let ${\bf w}=[k_1,\dots,k_r=k]\in \mathcal{T}^{\leq r}$. By $(2)_r$, there are $s$ and $t$ as in Theorem~\ref{thm: recursion for tropical signs}~$(b)$ corresponding to ${\bf w}$.
We will prove the claims for each case in the following:
\begin{itemize}
\item $k=1,2,3$.
\item $(\varepsilon_1^{\bf w},\varepsilon_{2}^{\bf w},\varepsilon_3^{\bf w})=(\pm 1,\pm1,\pm1)$.
\item $\mathrm{sign}(B^{\bf w})=\pm\mathrm{sign}(B)$.
\end{itemize}
However, some cases do not occur. For example, assume that $k=1$ and $\mathrm{sign}(B^{\bf w})=\mathrm{sign}(B)$. Then, by $(2)_r$, $(\varepsilon_1^{\bf w},\varepsilon_{2}^{\bf w},\varepsilon_3^{\bf w})=(1,1,1),(-1,-1,-1),(\pm1,-1,1)$ will not occur because there is no $s$ such that $\varepsilon_{s}^{\bf w}b_{s1}^{\bf w}>0$ and $\varepsilon_s^{\bf w} \neq \varepsilon_{1}^{\bf w}$. So, we need to check the cases that $(\varepsilon_1^{\bf w},\varepsilon_{2}^{\bf w},\varepsilon_3^{\bf w})=(1,1,-1),(1,-1,-1),(-1,1,1),(-1,1,-1)$. All cases that we cannot eliminate by the condition $(2)_r$ may be shown by a similar argument. 

Therefore, we focus on proving the case of $k=1$, $\mathrm{sign}(B^{\bf w})=\mathrm{sign}(B)$, and $(\varepsilon_1^{\bf w},\varepsilon_{2}^{\bf w},\varepsilon_3^{\bf w})=(1,1,-1)$. Namely, consider the case that
\begin{equation}
\left(\begin{matrix}
B^{\bf w}\\\hline
C^{\bf w}
\end{matrix}\right)
=\left(\begin{matrix}
0 & -w^{\bf w} & y^{\bf w}\\
z^{\bf w} & 0 & -u^{\bf w}\\
-v^{\bf w} & x^{\bf w} & 0\\\hline
\underset{+}{{\bf c}_{1}^{\bf w}} & \underset{+}{{\bf c}_{2}^{\bf w}} & \underset{-}{{\bf c}_{3}^{\bf w}}
\end{matrix}\right).
\end{equation}
Under this circumstance, the indices $s,t \in \{2,3\}$ defined in Theorem~\ref{thm: recursion for tropical signs} are $s=3$ and $t=2$. Then, we have
\begin{equation}
\begin{aligned}
\left(\begin{matrix}
B^{{\bf w}[2]}\\\hline
C^{{\bf w}[2]}
\end{matrix}\right)
&=\left(\begin{matrix}
0 & w^{\bf w} & -y^{{\bf w}[2]}\\
-z^{\bf w} & 0 & u^{\bf w}\\
v^{{\bf w}[2]} & -x^{\bf w} & 0\\\hline
{{\bf c}_{1}^{\bf w}+z^{\bf w}{\bf c}_2^{\bf w}} & {-{\bf c}_{2}^{\bf w}} & {{\bf c}_{3}^{\bf w}}
\end{matrix}\right),\\
\left(\begin{matrix}
B^{{\bf w}[3]}\\\hline
C^{{\bf w}[3]}
\end{matrix}\right)
&=\left(\begin{matrix}
0 & w^{{\bf w}[3]} & -y^{{\bf w}}\\
-z^{{\bf w}[3]} & 0 & u^{\bf w}\\
v^{{\bf w}} & -x^{\bf w} & 0\\\hline
{{\bf c}_{1}^{\bf w}+v^{\bf w}{\bf c}_3^{\bf w}} & {{\bf c}_{2}^{\bf w}} & {-{\bf c}_{3}^{\bf w}}
\end{matrix}\right).
\end{aligned}
\end{equation}

First, we show $(4)_{r}$.
By ${\bf c}_1^{\bf w},{\bf c}_2^{\bf w} \geq {\bf 0}$ and $z^{\bf w}
>0$, we have ${\bf c}^{{\bf w}[2]}_1={\bf c}_{1}^{\bf w}+z^{\bf w}{\bf c}_2^{\bf w} \geq {\bf c}^{\bf w}_1 \geq {\bf 0}$. Here, the first inequality is not an equality because ${\bf c}_2^{\bf w} \neq {\bf 0}$. Note that ${\bf c}_2^{\mfw[2]}=-{\bf c}_2^{\bf w}$ and ${\bf c}_3^{\mfw[2]}={\bf c}_3^{\bf w}$. Hence, the claim holds for ${\mfw}[2]$. As for $\mfw[3]$, we aim to prove that $-{\bf c}_1^{{\bf w}[3]}=-({\bf c}_{1}^{\bf w}+v^{\bf w}{\bf c}_3^{\bf w}) \geq {\bf c}_1^{\bf w}$. This inequality is equivalent to $2{\bf c}_1^{\bf w}+v^{\bf w}{\bf c}_3^{\bf w} \leq {\bf 0}$. This inequality follows from $(3)_r$, that is,
\begin{equation}
(y^{\bf w}v^{\bf w}-2){\bf c}_1^{\bf w}+v^{\bf w}{\bf c}_3^{\bf w} \leq {\bf 0}.
\end{equation}
Note that $2{\bf c}_1^{\bf w}+v^{\bf w}{\bf c}_{3}^{\bf w} \leq (y^{\bf w}v^{\bf w}-2){\bf c}_1^{\bf w}+v^{\bf w}{\bf c}_3^{\bf w}$ because $y^{\bf w}v^{\bf w} \geq 4$ and ${\bf c}_1^{\bf w} \geq {\bf 0}$. Thus, we obtain $2{\bf c}_1^{\bf w}+v^{\bf w}{\bf c}_3^{\bf w} \leq {\bf 0}$. Here, $2{\bf c}_{1}^{\bf w}+v^{\bf w}{\bf c}_{3}^{\bf w} \neq {\bf 0}$ holds because, by \Cref{prop: fundamental properties under sign-coherency}, $\mathbf{c}_1^{\mathbf{w}}$ and $\mathbf{c}_3^{\mathbf{w}}$ are linearly independent.

Furthermore, it also implies that $(1)_{r+1}$ holds. 

Next, we show $(2)_{r+1}$. For ${\bf w}[2]$, we have $k=2$ and then we choose $s=1$. For ${\bf w}[3]$, we have $k=3$ and then we choose $s=1$. It can be checked directly that $(2)_{r+1}$ holds. 

Last, we show $(3)_{r+1}$. For ${\bf w}[2]$, since $z^{\bf w}w^{\bf w}\geq 4$ and $\varepsilon^{\mfw[2]}_{1}=1$, we have
\begin{equation}
(z^{\bf w}w^{\bf w}-2){\bf c}_2^{{\bf w}[2]}+w^{\bf w}{\bf c}_1^{{\bf w}[2]} = 2{\bf c}_2^{\bf w}+w^{\bf w}{\bf c}_1^{\bf w} \geq {\bf 0}
\end{equation}
by ${\bf c}_1^{\bf w},{\bf c}_2^{\bf w} \geq {\bf 0}$ and $w^{\bf w} \geq 0$. For ${\bf w}[3]$, since $y^{\bf w}v^{\bf w}\geq 4$ and $\varepsilon^{\mfw[3]}_{1}=1$, we have
\begin{equation}
(y^{\bf w}v^{\bf w}-2){\bf c}_{3}^{{\bf w}[3]}+y^{\bf w}{\bf c}_1^{{\bf w}[3]} = 2{\bf c}_3^{\bf w}+y^{\bf w}{\bf c}_1^{\bf w}.
\end{equation}
Let ${\bf w}'=[k_1,\dots,k_{r-1}] \in \mathcal{T}^{\leq r-1}$, which is the sequence obtained by eliminating the last index $k=1$ from ${\bf w}$. Consider
\begin{equation}
C^{{\bf w}'}=C^{{\bf w}[1]}=\left(\begin{matrix}
-{\bf c}_1^{\bf w}, & {\bf c}_2^{\bf w}, & {\bf c}_3^{\bf w}+y^{\bf w}{\bf c}_1^{\bf w}
\end{matrix}\right).
\end{equation}
Then, by $(4)_{r-1}$, we have ${\varepsilon}_{3}^{{\bf w}'}({\bf c}_3^{\bf w}+y^{\bf w}{\bf c}_1^{\bf w}) \leq -{\bf c}_3^{\bf w}$. If $\varepsilon_3^{{\bf w}'}=1$, it implies $2{\bf c}_3^{\bf w}+y^{\bf w}{\bf c}_1^{\bf w} \leq {\bf 0}$ as we desired. If $\varepsilon_3^{{\bf w}'}=-1$, we obtain that
\begin{equation}
2{\bf c}_3^{\bf w}+y^{\bf w}{\bf c}_1^{\bf w} ={\bf c}_3^{\bf w}+\{{\bf c}_3^{\bf w}+y^{\bf w}{\bf c}_1^{\bf w}\} \leq {\bf c}_3^{\bf w} \leq {\bf 0},
\end{equation}
where the first inequality follows from ${\bf c}_3^{\bf w}+y^{\bf w}{\bf c}_1^{\bf w}={\bf c}_{3}^{{\bf w}'}\leq {\bf 0}$. Thus, $(3)_{r+1}$ holds.

In conclusion, all the claims $(1)$, $(2)$, $(3)$ and $(4)$ hold by induction.
\end{proof}

By the inequality (\ref{eq: k,s inequality}), we may derive a stronger inequality than (\ref{eq: monotonicity}) for some direction.
\begin{corollary}
Let $B \in \mathrm{M}_3(\mathbb{R})$ be a cluster-cyclic matrix.
Let ${\bf w} \in \mathcal{T} \setminus \{\emptyset\}$ and set $k$ as the last index of ${\bf w}$. Let $s \in \{1,2,3\}\backslash\{k\}$ be the index that satisfies (\ref{eq: choice of s}). Then, we have
\begin{equation}
\varepsilon_{k}^{{\bf w}[s]}{\bf c}_k^{{\bf w}[s]} \geq (|b_{sk}^{\bf w}b_{ks}^{\bf w}|-3)\varepsilon_{k}^{\bf w}{\bf c}_{k}^{\bf w} .
\end{equation}
\end{corollary}
\begin{proof}
We may verify ${\bf c}_{k}^{{\bf w}[s]}={\bf c}_k^{\bf w}+|b^{\bf w}_{sk}|{\bf c}_{s}^{\bf w}$ by the definition of $s$. By Theorem~\ref{thm: recursion for tropical signs}, we have $\varepsilon_{k}^{{\bf w}[s]}=-\varepsilon_k^{\bf w}=\varepsilon_s^{\bf w}$. So, we obtain
\begin{equation}
\varepsilon_{k}^{{\bf w}[s]}{\bf c}_k^{{\bf w}[s]}=-\varepsilon_{k}^{\bf w}{\bf c}_k^{\bf w}+|b^{\bf w}_{sk}|\varepsilon_s^{\bf w}{\bf c}_{s}^{\bf w}.
\end{equation}
By (\ref{eq: k,s inequality}), we have $|b_{sk}^{\bf w}|\varepsilon_s^{\bf w}{\bf c}_{s}^{\bf w} \geq (|b_{sk}^{\bf w}b_{ks}^{\bf w}|-2)\varepsilon_k^{\bf w}{\bf c}_{k}^{\bf w}$. Thus, we obtain $\varepsilon_k^{{\bf w}[s]}{\bf c}_k^{{\bf w}[s]} \geq (|b_{sk}^{\bf w}b_{ks}^{\bf w}|-3)\varepsilon_{k}^{\bf w}{\bf c}_k^{\bf w}$.
\end{proof}
\begin{example}\label{counter}
	Note that sign-coherence does not always hold for any real exchange matrices. The following is a counter-example to sign-coherence when the exchange matrix is not cluster-cyclic. Let the initial exchange matrix and $C$-matrix be 
	\begin{align}
		B^{\emptyset}=\left(\begin{matrix}
0 & 0.5 & 1\\
-0.5 & 0 & 0.5\\
-1 & -0.5 & 0
\end{matrix}\right),\ C^{\emptyset}=\left(\begin{matrix}
1 & 0 & 0\\
0 & 1 & 0\\
0 & 0 & 1
\end{matrix}\right).
	\end{align} Take $\mfw=[2,1,2,1]$ and then we have 
	\begin{align}
		B^{\mfw}=\left(\begin{matrix}
0 & 0.5 & 1.1875\\
-0.5 & 0 & -0.125\\
-1.1875 & 0.125 & 0
\end{matrix}\right),\ C^{\mfw}=\left(\begin{matrix}
1 & -0.5 & 0.0625\\
0.5 & 0.75 & -0.0938\\
0 & 0 & 1
\end{matrix}\right).
	\end{align} It implies that $C^{\mfw}$ is not sign-coherent.
\end{example}

\begin{remark}
	The conjecture of sign-coherence of $C$-matrices for ordinary cluster algebras was given by Fomin-Zelevinsky, see \cite[Prop. 5.6]{FZ07}. The sign-coherence conjecture was proved in the skew-symmetric case by \cite{DWZ10, Pla11, Nag13} with the algebraic representation methods. In general, the sign-coherence conjecture for the skew-symmetrizable case was proved by \cite{GHKK18} with the scattering diagram method. Here, although we only generalize and consider the rank $3$ real cluster-cyclic case, we introduce a new method as in \Cref{thm: recursion for tropical signs} and \Cref{lem: inequalities} to prove it.
\end{remark}

\section{Quadratic equations arising from $c$-vectors} \label{Sec: quadratic}
The relations between cluster variables and quadratic equations are studied by \cite{Lam16, GM23, CL24, CL25}. In this section, as an application of the sign-coherence and monotonicity of $c$-vectors for rank $3$ real cluster-cyclic exchange matrices, we exhibit the relations between the $c$-vectors and quadratic equations. Beforehand, we recall some basic notions and properties based on \cite{BGZ06, Sev11} as follows.
\begin{definition}[{\emph{real quasi-Cartan matrix}}]\label{quasi-Cartan matrix}
	A real matrix $A=(a_{ij})_{n\times n}\in \mathrm{M}_n(\mathbb{R})$ is called \emph{symmetrizable} if there exists a positive diagonal matrix $D=\diag(d_1,\dots,d_n)$ such that $DA$ is symmetric, that is $d_ia_{ij}=d_{j}a_{ji}$ for any $i,j$. Furthermore, $A$ is said to be a \emph{real quasi-Cartan matrix} if it is symmetrizable and all of its diagonal entries are $2$.
\end{definition}
Then, the quasi-Cartan matrices are related to skew-symmetrizable matrices by the following notion.
\begin{definition}[{\emph{real quasi-Cartan companion}}]
	Let $B=(b_{ij})_{n\times n}\in \mathrm{M}_n(\mathbb{R})$ be a real  skew-symmetrizable matrix. A \emph{real quasi-Cartan companion} of $B$ is a real quasi-Cartan matrix $A(B)=(a_{ij})_{n\times n}$, where $|a_{ij}|=|b_{ij}|$ for all $i\neq j$. Note that $A(B)$ is a symmetrizable matrix with the symmetrizer $D$ as $B$.
\end{definition}
\begin{remark}
	In \cite{LL24}, when $n=3$, if $a_{ij}=|b_{ij}|$ for all $i\neq j$, then $A(B)$ is called the \emph{pseudo-Cartan companion} of $B$.
\end{remark}
Now, we focus on the case of rank $3$ and give some new notions as follows. 
\begin{definition}\label{qCc}
	Let $B \in \mathrm{M}_3(\mathbb{R})$ be a cluster-cyclic exchange matrix. Define a set $\mcP^{\mfw}$ associated to the reduced mutation sequence $\mfw$ as follows:
	\begin{align}\label{Pw}
		\mcP^{\mfw}=\{(i,j)|\ \varepsilon^{\mfw}_{i}b^{\mfw}_{ij}>0\ \text{or}\ \varepsilon^{\mfw}_{j}b^{\mfw}_{ji}>0\}.
	\end{align}  Note that if $(i,j)\in \mcP^{\mfw}$, we also have $(j,i)\in \mcP^{\mfw}$.
\end{definition}
In the following, we introduce the notion of real quasi-Cartan companion associated to $\mcP^{\mfw}$.
\begin{definition}[{\emph{quasi-Cartan companion of $\mcP^{\mfw}$}}] 
\label{quasi-Cartan companion of P} Let $B \in \mathrm{M}_3(\mathbb{R})$ be a cluster-cyclic exchange matrix with the skew-symmetrizer $D$ and $\mfw$ be a reduced mutation sequence.\

\begin{enumerate}
	\item 
For $\mfw=\emptyset$ and $q\in \{1,2,3\}$, we denote by $A_q$ the quasi-Cartan companion of $B$ as follows: \begin{align}
	A^{\emptyset}_q=A_q=\begin{pmatrix}
		2 & \varepsilon^q_{12}|b_{12}| & \varepsilon^q_{13}|b_{13}|\\ \varepsilon^q_{21}|b_{21}|& 2 & \varepsilon^q_{23}|b_{23}| \\ \varepsilon^q_{31}|b_{31}| & \varepsilon^q_{32}|b_{32}| & 2	\end{pmatrix},
\end{align} where for any $i\neq j$, $$
	\varepsilon^q_{ij}= \left\{
		\begin{array}{ll}
			1, &  \text{if}\ q\notin \{i,j\},  \\
			-1, &  \text{if}\  q \in \{i,j\}. 
		\end{array} \right.
$$  
\item For $|\mfw|\geq 1$, let $A^{\mfw}=(a^{\mfw}_{ij})_{3\times 3}$ be the quasi-Cartan companion of $B^{\mfw}$  as follows:
	\begin{align}
		a^{\mfw}_{ij}=\left\{
		\begin{array}{ll}
			2, &  \text{if}\ i=j;  \\
			-|b^{\mfw}_{ij}|, &  \text{if}\ i\neq j \;\;
			\mbox{and}\; (i,j) \in \mcP^{\mfw}; \\ |b^{\mfw}_{ij}|, &  \text{if}\ i\neq j \;\;
			\mbox{and}\; (i,j) \notin  \mcP^{\mfw}.
		\end{array} \right.
	\end{align}
	\end{enumerate} We call $A_q\ (q=1,2,3)$ and $A^{\mfw}(|\mfw|\geq 1)$ the \emph{quasi-Cartan companion of $\mcP^{\mfw}$}. Furthermore, we denote by $\tilde{A_q}=DA_q$ and $\tilde{A}^{\mfw}=DA^{\mfw}$. Then, we call them the \emph{symmetrized quasi-Cartan companion of $\mcP^{\mfw}$}.
\end{definition} 

\begin{lemma}
	Let $B \in \mathrm{M}_3(\mathbb{R})$ be a cluster-cyclic exchange matrix. Then, for any  reduced mutation sequence $\mfw$ with $|\mfw|\geq 1$, there are exactly four elements in $\mcP^{\mfw}$. 
\end{lemma}
\begin{proof}
	By \Cref{two signs}, the three tropical signs $\varepsilon^{\mfw}_{1},\varepsilon^{\mfw}_{2},\varepsilon^{\mfw}_{3}$ are not all equal. Let us consider the case where $\varepsilon^{\mfw}_{l}=\varepsilon^{\mfw}_{m}=1$ and $\varepsilon^{\mfw}_{n}=-1$ for pairwise distinct $l, m, n \in \{1, 2, 3\}$. (The case that $\varepsilon^{\mfw}_{l}=\varepsilon^{\mfw}_{m}=-1$ and $\varepsilon^{\mfw}_{n}=1$ is similar).
	
	If $b^{\mfw}_{lm}>0$, since $B^{\mfw}$ is cyclic, then we have $b^{\mfw}_{mn}>0$ and $b^{\mfw}_{nl}>0$. It implies that 
	\begin{equation}
	\begin{array}{cc}
		\varepsilon^{\mfw}_{l}b^{\mfw}_{lm}>0,\ \varepsilon^{\mfw}_{m}b^{\mfw}_{mn}>0,\ \varepsilon^{\mfw}_{n}b^{\mfw}_{nm}>0, \\ \varepsilon^{\mfw}_{l}b^{\mfw}_{ln} <0,\ \varepsilon^{\mfw}_{m}b^{\mfw}_{ml}<0,\ \varepsilon^{\mfw}_{n}b^{\mfw}_{nl}<0.
		\end{array}
		  	\end{equation}
	 Hence, we have $\mcP^{\mfw}=\{(l,m),(m,l),(m,n),(n,m)\}$. 
     
	If $b^{\mfw}_{lm}<0$, since $B^{\mfw}$ is cyclic, then we have $b^{\mfw}_{mn}<0$ and $b^{\mfw}_{nl}<0$. It implies that \begin{equation}
	\begin{array}{cc}
		\varepsilon^{\mfw}_{l}b^{\mfw}_{ln}>0,\  \varepsilon^{\mfw}_{m}b^{\mfw}_{ml}>0,\  \varepsilon^{\mfw}_{n}b^{\mfw}_{nl}>0,\\ \varepsilon^{\mfw}_{l}b^{\mfw}_{lm}<0,\ \varepsilon^{\mfw}_{m}b^{\mfw}_{mn}<0,\ \varepsilon^{\mfw}_{n}b^{\mfw}_{nm}<0.
		\end{array}
		  	\end{equation}
Hence, we have $\mcP^{\mfw}=\{(l,n),(n,l),(l,m),(m,l)\}$. 

Then, there are exactly four elements in $\mcP^{\mfw}$ with $|\mfw|\geq 1$.
\end{proof}
\begin{remark}
	If $|\mfw|=0$, that is $\mfw=\emptyset$, it is clear that there are exactly six elements in the set $\mcP^{\mfw}$ defined by \eqref{Pw}. 
\end{remark}

\begin{theorem}\label{duality}
	Let $B \in \mathrm{M}_3(\mathbb{R})$ be a cluster-cyclic exchange matrix with the skew-symmetrizer $D$ and $\mfw=[k_1,\dots,k_r]$ be any reduced sequence with $r\geq 1$. Then, the following quasi-Cartan congruence relation holds:
	\begin{align}
		(C^{\mfw})^{\top}\tilde{A}_{k_1}C^{\mfw}=\tilde{A}^{\mfw}. \label{qCc}
	\end{align}
	\end{theorem}
\begin{proof}
 After relabelling the indices and replacing $B$ by $-B$ if necessary, without loss of generality, we may assume that $k_1=3$ and $\sign(b_{12})=\sign(b_{23})=\sign(b_{31})=1$. Then, we have 
	\begin{align}
		\tilde{A}_3=\begin{pmatrix}
		2d_1 & d_1|b_{12}| & -d_1|b_{13}|\\ d_2|b_{21}|& 2d_2 & -d_2|b_{23}| \\ -d_3|b_{31}| & -d_3|b_{32}| & 2d_3
	\end{pmatrix}.
	\end{align}
    We prove the claim by induction on the length of $\mfw$. If $|\mfw|=1$, we obtain that
    \begin{align}
		B^{[3]}= \begin{pmatrix}
		0 & b_{12}-b_{13}b_{32} & -b_{13}\\
        b_{21}+b_{23}b_{31}& 0 & -b_{23} \\
        -b_{31} & -b_{32} & 0
	   \end{pmatrix},\ 
        C^{[3]}= \begin{pmatrix}
		  1 & 0 & 0\\
        0& 1 & 0 \\
        b_{31} & 0 & -1
	   \end{pmatrix},
	\end{align}
    and then $\mcP^{[3]}=\{(1,3),(3,1), (1,2), (2,1)\}$. It can be checked directly that the equality \eqref{qCc} holds. Assume that the equality holds for $\mfw=[3,\dots, k_r]$, where $r\geq 1$ and $k_r=k\in \{1,2,3\}$. We need to prove that for any $i\in \{1,2,3\}\backslash\{k\}$, the following equality holds:
		\begin{align}
			(C^{\mfw[i]})^{\top}\tilde{A}_{k_1}C^{\mfw[i]}=\tilde{A}^{\mfw[i]}.
		\end{align}
	Note that $C^{\mfw[i]}=C^{\mfw}X^{i;\mfw}$, where we denote by $X^{i;\mfw}=J_i+[\varepsilon_i^{\bf w}B^{\bf w}]^{i \bullet}_{+}$. Hence, it suffices to prove that 
	\begin{align}
    (X^{i;\mfw})^{\top}\tilde{A}^{\mfw}X^{i;\mfw}=\tilde{A}^{\mfw[i]}.
	\end{align}	
	Here, we denote by $X^{i;\mfw}=(\mfx^{i;\mfw}_1,\mfx^{i;\mfw}_2,\mfx^{i;\mfw}_3)$, where $\mfx^{i;\mfw}_j\in \mbR^3$. 	By \Cref{thm: recursion for tropical signs}, there is a unique $s\in \{1,2,3\}\backslash\{k\}$, such that 
	\begin{align}
		\varepsilon_s^{\mfw}\mathrm{sign}(b_{ks}^{\mfw})<0,\ \varepsilon_k^{\mfw}\neq \varepsilon_s^{\mfw}.
	\end{align} Let $t\in \{1,2,3\}\backslash\{k,s\}$ be the other index. Without loss of generality, we may assume that $\varepsilon_s^{\mfw}>0$. Then, $b_{sk}^{\mfw}>0$ and $\varepsilon_k^{\mfw}<0$ hold. (The case that $\varepsilon_s^{\mfw}<0$ is similar.) Since $B^{\mfw}$ is cluster-cyclic, we have $b_{kt}^{\mfw}>0$ and $b_{ts}^{\mfw}>0$. There are two possible cases to be discussed.
	\begin{enumerate}[leftmargin=2em]\itemsep=0pt
		\item If $\varepsilon_t^{\mfw}>0$, then $\varepsilon_t^{\mfw}b_{ts}^{\mfw}>0$ and $(\varepsilon^{\mfw}_{s},\varepsilon^{\mfw}_{t},\varepsilon^{\mfw}_{k})=(1,1,-1)$. Hence, it implies that $\mcP^{\mfw}=\{(s,k),(k,s),(t,s),(s,t)\}$ and $(t,k),(k,t)\notin \mcP^{\mfw}$. 
	
		Firstly, let $i=s$ and we obtain that 
		\begin{align}
			\mfx^{s;\mfw}_s=-\mfe_s,\ \mfx^{s;\mfw}_t=\mfe_t,\ \mfx^{s;\mfw}_k=\mfe_k+b_{sk}^{\mfw}\mfe_s.
		\end{align} Furthermore, by \Cref{thm: recursion for tropical signs} $(b)$, we have $(\varepsilon^{\mfw[s]}_{s},\varepsilon^{\mfw[s]}_{t},\varepsilon^{\mfw[s]}_{k})=(-1,1,1)$. Note that $b_{st}^{\mfw[s]}>0$, $b_{tk}^{\mfw[s]}>0$ and $b_{ks}^{\mfw[s]}>0$. Hence,  $\mcP^{\mfw[s]}=\{(s,k),(k,s),(t,k),(k,t)\}$ and $(s,t),(t,s)\notin \mcP^{\mfw[s]}$. Since $(X^{s;\mfw})^{\top}\tilde{A}^{\mfw}X^{s;\mfw}=((\mfx^{s;\mfw}_{i})^{\top}\tilde{A}^{\mfw}\mfx^{s;\mfw}_{j})_{n\times n}$, it is direct that $(\mfx^{s;\mfw}_{s})^{\top}\tilde{A}^{\mfw}\mfx^{s;\mfw}_{s}=2d_s$ and $(\mfx^{s;\mfw}_{t})^{\top}\tilde{A}^{\mfw}\mfx^{s;\mfw}_{t}=2d_t$. Moreover, it implies that \begin{equation}\begin{array}{ll}
		(\mfx^{s;\mfw}_{k})^{\top}\tilde{A}^{\mfw}\mfx^{s;\mfw}_{k}&= (\mfe_k+b_{sk}^{\mfw}\mfe_s)^{\top}\tilde{A}^{\mfw}(\mfe_k+b_{sk}^{\mfw}\mfe_s)\\
			&= \mfe_k^{\top}\tilde{A}^{\mfw}\mfe_k+2b_{sk}^{\mfw}\mfe_s^{\top}\tilde{A}^{\mfw}\mfe_k+(b_{sk}^{\mfw})^2\mfe_s^{\top}\tilde{A}^{\mfw}\mfe_s=2d_k;\\
			(\mfx^{s;\mfw}_{s})^{\top}\tilde{A}^{\mfw}\mfx^{s;\mfw}_{k}&= -\mfe_s^{\top}\tilde{A}^{\mfw}(\mfe_k+b_{sk}^{\mfw}\mfe_s)=d_s|b_{sk}^{\mfw}|-2d_sb_{sk}^{\mfw}=-d_s|b_{sk}^{\mfw[s]}|;\\ (\mfx^{s;\mfw}_{t})^{\top}\tilde{A}^{\mfw}\mfx^{s;\mfw}_{k}&= \mfe_t^{\top}\tilde{A}^{\mfw}(\mfe_k+b_{sk}^{\mfw}\mfe_s)=-d_tb_{tk}^{\mfw}-d_tb_{ts}^{\mfw}b_{sk}^{\mfw}=-d_t|b_{tk}^{\mfw[s]}|;\\
			(\mfx^{s;\mfw}_{s})^{\top}\tilde{A}^{\mfw}\mfx^{s;\mfw}_{t}&= -\mfe_s^{\top}\tilde{A}^{\mfw}\mfe_{t}=d_s|b_{st}^{\mfw}| =d_s|b_{st}^{\mfw[s]}|.		\end{array}\end{equation}  Then, we have $(X^{s;\mfw})^{\top}\tilde{A}^{\mfw}X^{s;\mfw}=\tilde{A}^{\mfw[s]}$ and $
		(C^{\mfw[s]})^{\top}\tilde{A}_{k_1}C^{\mfw[s]}=\tilde{A}^{\mfw[s]}$.

	 Secondly, let $i=t$. Then, we obtain that 
		\begin{align}
			\mfx^{t;\mfw}_t=-\mfe_t,\ \mfx^{t;\mfw}_k=\mfe_k,\ \mfx^{s;\mfw}_s=\mfe_s+b_{ts}^{\mfw}\mfe_t.
		\end{align} Furthermore, by \Cref{thm: recursion for tropical signs} $(b)$, we have  $(\varepsilon^{\mfw[t]}_{s},\varepsilon^{\mfw[t]}_{t},\varepsilon^{\mfw[t]}_{k})=(1,-1,-1)$. Note that $b_{st}^{\mfw[t]}>0$, $b_{tk}^{\mfw[t]}>0$ and $b_{ks}^{\mfw[t]}>0$. Therefore,  $\mcP^{\mfw[t]}=\{(s,t),(t,s),(k,t),(t,k)\}$ and $(s,k),(k,s)\notin \mcP^{\mfw[t]}$. It is direct that $(\mfx^{t;\mfw}_{t})^{\top}\tilde{A}^{\mfw}\mfx^{t;\mfw}_{t}=2d_t$ and $(\mfx^{t;\mfw}_{k})^{\top}\tilde{A}^{\mfw}\mfx^{t;\mfw}_{k}=2d_k$. Moreover, it implies that  
		\begin{equation}\begin{array}{ll}
		(\mfx^{t;\mfw}_{s})^{\top}\tilde{A}^{\mfw}\mfx^{t;\mfw}_{s}&= (\mfe_s+b_{ts}^{\mfw}\mfe_t)^{\top}\tilde{A}^{\mfw}(\mfe_s+b_{ts}^{\mfw}\mfe_t)\\
			&= \mfe_s^{\top}\tilde{A}^{\mfw}\mfe_s+2b_{ts}^{\mfw}\mfe_t^{\top}\tilde{A}^{\mfw}\mfe_s+(b_{ts}^{\mfw})^2\mfe_t^{\top}\tilde{A}^{\mfw}\mfe_t=2d_s;\\
			(\mfx^{t;\mfw}_{s})^{\top}\tilde{A}^{\mfw}\mfx^{t;\mfw}_{t}&=-(\mfe_{s}+b_{ts}^{\mfw}\mfe_{t})^{\top}\tilde{A}^{\mfw}\mfe_{t}=d_s|b^{\mfw}_{st}|+2d_sb^{\mfw}_{st}=-d_s|b^{\mfw[t]}_{st}|; \\
				(\mfx^{t;\mfw}_{k})^{\top}\tilde{A}^{\mfw}\mfx^{t;\mfw}_{t}&=-\mfe_{k}^{\top}\tilde{A}^\mfw\mfe_{t}=-d_k|b^{\mfw}_{kt}|=-d_k|b^{\mfw[t]}_{kt}|; \\
				(\mfx^{t;\mfw}_{s})^{\top}\tilde{A}^{\mfw}\mfx^{t;\mfw}_{k}&=(\mfe_{s}+b_{ts}^{\mfw}\mfe_{t})^{\top}\tilde{A}^{\mfw}\mfe_{k}=d_k(b^{\mfw}_{ks}+b^{\mfw}_{kt}b^{\mfw}_{ts})=d_k|b^{\mfw[t]}_{ks}|.
		\end{array}\end{equation} Then, we have $(X^{t;\mfw})^{\top}\tilde{A}^{\mfw}X^{t;\mfw}=\tilde{A}^{\mfw[t]}$ and $
		(C^{\mfw[t]})^{\top}\tilde{A}_{k_1}C^{\mfw[t]}=\tilde{A}^{\mfw[t]}$.

		\item If $\varepsilon_t^{\mfw}<0$, then $\varepsilon_t^{\mfw}b_{tk}^{\mfw}>0$ and $(\varepsilon^{\mfw}_{s},\varepsilon^{\mfw}_{t},\varepsilon^{\mfw}_{k})=(1,-1,-1)$. Hence, it implies that $\mcP^{\mfw}=\{(s,k),(k,s), (t,k), (k,t)\}$ and $(t,s),(s,t)\notin \mcP^{\mfw}$. Then, the proof is similar to case $(1)$ and we may check it directly that $
		(C^{\mfw[s]})^{\top}\tilde{A}_{k_1}C^{\mfw[s]}=\tilde{A}^{\mfw[s]}$ and $
		(C^{\mfw[t]})^{\top}\tilde{A}_{k_1}C^{\mfw[t]}=\tilde{A}^{\mfw[t]}$.
	\end{enumerate} Hence, we complete the proof.
\end{proof}
\begin{remark}
	In \cite[Theorem 1.3]{Sev15}, it was proved that $c$-vectors associated with an acyclic seed of skew-symmetric cluster algebras  correspond to a quasi-Cartan companion of the exchange matrix. More precisely, $A^{\mfw}=((\mfc^{\mfw}_i)^{\top}A\mfc^{\mfw}_j)_{n\times n}$ is a quasi-Cartan companion of $B^{\mfw}$, where $A$ is the generalized Cartan matrix of the initial exchange matrix $B$.
\end{remark}

Now, we can directly get the following theorem about $c$-vectors and quadratic equations by \Cref{duality} if we focus on the diagonal entries of \eqref{qCc}.
\begin{theorem}\label{main equation}
	Let $B \in \mathrm{M}_3(\mathbb{R})$ be a cluster-cyclic exchange matrix with the skew-symmetrizer $D$ and $\mfw=[k_1,\dots,k_r]$ be any reduced mutation sequence with $|\mfw|\geq 1$. Then, we have 
	\begin{enumerate}
		\item[$(1)$] If $k_1=1$, then  $\mfc^{\mfw}_{i}$ are solutions to the quadratic equations: $$d_1x_1^2+d_2x_2^2+d_3x_3^2-d_1|b_{12}|x_1x_2+d_2|b_{23}|x_2x_3-d_3|b_{31}|x_1x_3=d_i.\ (i=1,2,3)$$
		\item[$(2)$] If $k_1=2$, then $\mfc^{\mfw}_{i}$ are solutions to the quadratic equations: $$d_1x_1^2+d_2x_2^2+d_3x_3^2-d_1|b_{12}|x_1x_2-d_2|b_{23}|x_2x_3+d_3|b_{31}|x_1x_3=d_i.\ (i=1,2,3)$$

		\item[$(3)$] If $k_1=3$, then  $\mfc^{\mfw}_{i}$ are solutions to the quadratic equations: $$d_1x_1^2+d_2x_2^2+d_3x_3^2+d_1|b_{12}|x_1x_2-d_2|b_{23}|x_2x_3-d_3|b_{31}|x_1x_3=d_i.\ (i=1,2,3)$$
	\end{enumerate}
In conclusion, all the $c$-vectors $\mfc^{\mfw}_{i}(i=1,2,3)$ are solutions to \begin{align}
d_1x_1^2+d_2x_2^2+d_3x_3^2+\varepsilon_1^{[k_1]}\varepsilon_2^{[k_1]}d_1|b_{12}|x_1x_2+\varepsilon_2^{[k_1]}\varepsilon_3^{[k_1]}d_2|b_{23}|x_2x_3+\varepsilon_1^{[k_1]}\varepsilon_3^{[k_1]}d_3|b_{31}|x_1x_3=d_i. \label{big equ}\end{align}
\end{theorem}
In addition, it is clear that the initial $c$-vectors $\mfe_i\ (i=1,2,3)$ are also solutions to the quadratic equations as above. Hence, we conclude that every $c$-vector of a rank $3$ real cluster-cyclic exchange matrix is a solution to some quadratic equation.
\begin{proposition}\label{prop: conjectures are true}
Conjecture~\ref{conjecture: for real entries} holds for any cluster-cyclic exchange matrix $B \in \mathrm{M}_{3}(\mathbb{R})$ of rank $3$.
\end{proposition}
\begin{proof}
The claim ($a$) follows from Theorem~\ref{thm: sign-coherency}. 
Now, we focus on proving ($b$). Suppose that ${\bf c}_{i}^{\bf w}=\alpha{\bf e}_j$. By substituting this into the equality \eqref{big equ} of Theorem~\ref{main equation}, we obtain $d_j\alpha^{2}=d_i$, regardless of the choice of the initial mutation. Hence, this implies $\alpha=\pm\sqrt{d_id_j^{-1}}$ as we desired.
\end{proof}
In fact, when we consider integer $G$-matrices of ordinary cluster algebras, we often focus the row sign-coherent property, see \cite{NZ12,GHKK18}. Based on \Cref{thm: sign-coherency} and \Cref{prop: conjectures are true}, for the real cluster-cyclic cases of rank $3$, we can also obtain the row sign-coherence for real $G$-matrices by the \emph{third duality}, see \cite[Proposition 6.1]{AC25}. However, since we only focus on the structures and properties of the $C$-pattern here, we will not introduce many details about the $G$-pattern.
\begin{remark}\label{Lee rmk}
	In \cite{EJLN24}, it was proved that for a cluster-cyclic quiver $Q$ with $3$ vertices, every $c$-vector is a solution to a quadratic equation of the form:
	\begin{align}
		\sum_{i=1}^{3} x^2_i+\sum_{1\leq i<j\leq 3} \pm q_{ij}x_ix_j=1,
	\end{align}
    where $q_{ij}$ is the number of arrows between the vertices $i$ and $j$ in $Q$. However, it is difficult to determine the exact signs of the terms $x_ix_j$. Now, according to \Cref{main equation}, we give the answer to this problem and generalize it to the real cluster-cyclic skew-symmetrizable case. In fact, the exact signs are  determined by the direction of the first cluster mutation.
\end{remark}
\begin{example}\label{equa exam}
	Let $B=\small \begin{pmatrix}
		0 & 2 & -4\\ -3& 0 & 6 \\ 2 & -2 & 0
	\end{pmatrix}$ be a cluster-cyclic exchange matrix with the skew-symmetrizer $D=\diag(3,2,6)$. Take the mutation sequence $\mfw=[3,2,1]$. Then, we have \begin{align}
		B^{[321]}=\small \begin{pmatrix}
		0 & -6 & 32\\ 9& 0 & -282 \\ -16 & 94 & 0
	\end{pmatrix},\ 
C^{[321]}=\small \begin{pmatrix}
		-1 & 6 & 0\\ -9& 53 & 0 \\ -2 & 12 & -1
	\end{pmatrix},
	\end{align} and $
		\mcP^{[321]}=\{(1,2),(2,1), (1,3), (3,1)\}.$
		It implies that \begin{align}
		A^{[321]}=\small \begin{pmatrix}
		2 & -6 & -32\\ -9& 2 & 282 \\ -16 & 94 & 2
	\end{pmatrix},\	\tilde{A}^{[321]}=\small \begin{pmatrix}
		6 & -18 & -96\\ -18& 4 & 564 \\ -96 & 564 & 12
	\end{pmatrix}.
		\end{align} Furthermore, we can directly check that 
		\begin{align}
			(C^{[321]})^{\top}\tilde{A}_{3}C^{[321]}&=\small \begin{pmatrix}
		-1 & -9 & -2\\ 6& 53 & 12 \\ 0 & 0 & -1
	\end{pmatrix} \small \begin{pmatrix}
		6 & 6 & -12\\ 6& 4 & -12 \\ -12 & -12 & 12
	\end{pmatrix}\small\begin{pmatrix}
		-1 & 6 & 0\\ -9& 53 & 0 \\ -2 & 12 & -1
	\end{pmatrix}\\ &= \small \begin{pmatrix}
		6 & -18 & -96\\ -18& 4 & 564 \\ -96 & 564 & 12
	\end{pmatrix}.
		\end{align} Hence, we conclude that $(C^{[321]})^{\top}\tilde{A}_{3}C^{[321]}=\tilde{A}^{[321]}$.
\end{example}

\section{Exchange graphs corresponding to cluster-cyclic exchange 
 matrices}\label{sec: proof of the exchange graph}
In this section, we aim to show the $3$-regular tree structure of the exchange graphs of $C,G$-patterns associated with rank $3$ real cluster-cyclic exchange matrices.

Fix a cluster-cyclic exchange matrix $B \in \mathrm{M}_3(\mathbb{R})$, we write each $C$-matrix as $C^{\bf w}=(c_{ij}^{\bf w})_{3 \times 3}$. Then, we define its \emph{complexity} as follows:
\begin{equation}
\kappa(C^{\bf w})=\sum_{1\leq i,j \leq 3}|c_{ij}^{\bf w}|.
\end{equation} Moreover, by the notion of an unlabeled $C$-matrix, we can define its complexity as $\kappa([C^{\bf w}])=\kappa(C^{\bf w})$ and it is well-defined.

Now, we introduce a partial order $\leq$ on the mutation sequences $\mathcal{T}$ as follows:
\begin{equation}\label{eq: partial order on T}
{\bf w}=[k_1,\dots,k_r] \leq {\bf u}=[l_1,\dots,l_{r'}]\ \Longleftrightarrow\ r\leq r^{\prime}\ \&\ \textup{$k_i=l_i$ for any $i=1,\dots,r$.}
\end{equation} Furthermore, if $\mfw \leq \mfu$ and $\mfw \neq \mfu$, then we denote by $\mfw < \mfu$. 
Then, $\kappa$ has the following property.
\begin{lemma}\label{lem: monotonicity of kappa}
Let $B \in \mathrm{M}_3(\mathbb{R})$ be a cluster-cyclic exchange matrix. For any ${\bf w},{\mfu} \in \mathcal{T}$, if ${\bf w} < {\bf u}$, then  $\kappa(C^{\bf w})<\kappa(C^{\mfu})$.
\end{lemma}
\begin{proof}
This is immediately shown by Lemma~\ref{lem: inequalities}~$(b)$.
\end{proof}
\begin{theorem}\label{thm: exchange graph C}
For any cluster-cyclic exchange matrix $B \in \mathrm{M}_3(\mathbb{R})$, the exchange graphs of the $C$-pattern $\mfE \mfG(\mfC)$ and the exchange graph of the $G$-pattern $\mfE \mfG(\mfG)$ are $3$-regular trees.
\end{theorem}
\begin{proof}
By Proposition~\ref{prop: two exchange graphs are the same} and Proposition~\ref{prop: conjectures are true}, it suffices to show the case of the $C$-pattern. 
Firstly, by the definition of the mutations and unlabeled $C$-matrices, we may show that the degree of each vertex is $3$. Moreover, $\mfE \mfG(\mfC)$ is connected because every vertex can be reached from the initial vertex by a finite sequence of mutations. By \Cref{lem: monotonicity of kappa}, we conclude that there are no loops in $\mfE \mfG(\mfC)$.

Secondly, we focus on proving that there are no finite cycles with length no less than $3$ in the exchange graph $\mfE \mfG(\mfC)$. Suppose that there is a cycle $\mcO$ in $\mfE \mfG(\mfC)$. Take one vertex $[C^{\bf w}] \in \mcO$ such that $\kappa([C^{\bf w}])$ is the largest. Since $\mcO$ is a cycle with length no less than $3$, there are two indices $i,j\in \{1,2,3\}$, such that $C^{{\bf w}[i]},C^{{\bf w}[j]} \in \mcO$. Without loss of generality, we may assume that ${\bf w}<{\bf w}[i]$. (Note that if ${\bf w}[i]<{\bf w}$, then $j$ should satisfy ${\bf w}<{\bf w}[j]$.) Then, by Lemma~\ref{lem: monotonicity of kappa}, we have $\kappa([C^{\bf w}])<\kappa([C^{{\bf w}[i]})]$, which contradicts the fact that $\kappa([C^{\bf w}])$ is the largest in $\mcO$. Hence, we conclude that $\mfE \mfG(\mfC)$ is a $3$-regular tree.
\end{proof}
Here, we proved \Cref{thm: exchange graph C} based on the complexity of $C$-matrices, which is motivated by Tomoki Nakanishi.

\begin{remark}
	In \cite{War14}, it was proved that the exchange graph of a cluster-cyclic quiver with three vertices is a $3$-regular tree by using the method of \emph{fork}, which may not be suitable for general cases. Here, we use the monotonicity of $c$-vectors and generalize it to the real cluster-cyclic skew-symmetrizable case.
\end{remark}
Note that if we restrict to the ordinary (integer) cluster algebras, these three kinds of exchange graphs $\mfE \mfG(\bf{\Sigma}), \mfE \mfG(\mfC)$ and $\mfE \mfG(\mfG)$ are isomorphic by \Cref{Synchronicity Nak}. In particular, based on \Cref{thm: exchange graph C}, we can immediately get the following corollary.
\begin{corollary}\label{cor: exchange graph}
For the cluster-cyclic skew-symmetrizable cluster algebra $\mcA$ of rank $3$, the following statements hold:
\begin{enumerate}
	\item[$(1)$] The exchange graph of the cluster pattern $\mfE \mfG(\bf{\Sigma})$ is a $3$-regular tree.
	\item[$(2)$] The exchange graph of the $C$-pattern $\mfE \mfG(\mfC)$ is a $3$-regular tree.
	\item[$(3)$] The exchange graph of the $G$-pattern $\mfE \mfG(\mfG)$ is a $3$-regular tree.
\end{enumerate}
	
\end{corollary}

\begin{remark}
As explained in \Cref{rem: bad example of g-vector fan}, in the real setting, the dual graph of the $g$-vector fan might be different from the above exchange graphs. However, with a slight modification, we can conclude that it is also a $3$-regular tree: We fix a skew-symmetrizer $D=\mathrm{diag}(d_1,d_2,d_3)$ of $B$, and we define the \emph{modified $C$-matrices} $\tilde{C}^{\mathbf{w}}$ as $\tilde{C}^{\mathbf{w}} = C^{\mathbf{w}}D^{-1/2}$, where $D^{-1/2}=\mathrm{diag}(\sqrt{1/d_1},\sqrt{1/d_2},\sqrt{1/d_3})$. Then, we can define the \emph{exchange graph of the unlabeled modified $C$-matrices} in the same way as \Cref{CG-EG}. In \cite[Thm.~10.7]{AC25}, it was shown that, under \Cref{conjecture: for real entries}, the dual graph of the $g$-vector fan coincides with the exchange graph of the unlabeled modified $C$-matrices. Since the same proof as in \Cref{thm: exchange graph C} works for the modified $C$-matrices, the dual graph of the $g$-vector fan is indeed a $3$-regular tree.
\end{remark}

\section{Structures of tropical signs} \label{sec: Structure of tropical signs}
In this section, we exhibit the novel structure and properties of tropical signs for rank $3$ real cluster-cyclic exchange matrices.
\subsection{The binary components of tropical signs}\

Recall that we define the partial order $\leq$ on $\mathcal{T}$ by \eqref{eq: partial order on T}.
For any ${\bf w} \in \mathcal{T}$, we denote by 
\begin{equation}
\mathcal{T}^{\geq {\bf w}}=\{{\bf u} \in \mathcal{T} \mid {\bf w} \leq {\bf u}\}.
\end{equation}
In particular, a subset $\mathcal{T}^{\geq [i]}$ ($i=1,2,3$) is called a {\em subtree on the direction $i$}.
To simplify the statements, we fix an initial direction of mutation $i=1,2,3$, and we define the maps $K,S,T:\mathcal{T}\setminus\{\emptyset\} \to \{1,2,3\}$ by $K({\bf w})=k$, $S({\bf w})=s$, and $T({\bf w})=t$, where $k,s,t \in \{1,2,3\}$ are the indices defined in Theorem~\ref{thm: recursion for tropical signs}. We introduce the following notations in $\mathcal{T}^{\geq [i]}$ instead of the mutation ${\bf w}[l]$.
\begin{definition}\label{def: monoid action of M}
Let $\mathcal{M}$ be a free monoid generated by $S$ and $T$. (Here, we view the maps $S,T:\mathcal{T}\backslash\{\emptyset\}\to\{1,2,3\}$ as formal symbols.) Each element of $\mathcal{M}$ is called a {\em word}. Fix a direction $i=1,2,3$ of the initial mutation, and we introduce a right monoid action of $\mathcal{M}$ on $\mathcal{T}^{\geq [i]}$ as
\begin{equation}
{\bf w}M={\bf w}[M({\bf w})] \quad ({\bf w} \in \mathcal{T}^{\geq [i]},\ M=S,T).
\end{equation} 
\end{definition}
For any word $X=M_1\cdots M_l \in \mathcal{M}$ with $M_i =S,T$, a {\em subword} of $X$ is a word of the form $M_{k}M_{k+1}\cdots M_{r}$ for some $1\leq k \leq r \leq l$. In particular, every subword which starts from the initial letter (namely, $k=1$) is called a {\em prefix} of $X$.
\par
We write the {\em length} of a word $X=M_1M_2\cdots M_l \in \mathcal{M}$ ($M_i = S,T$) as $|X|=l$. Then, for any ${\bf w} \in \mathcal{T}\setminus\{\emptyset\}$, since $K({\bf w}) \neq M({\bf w})$ ($M=S,T$), we have
\begin{equation}
|{\bf w}X|=|{\bf w}|+|X|,
\end{equation}
where $|{\bf w}|$ and $|{\bf w}X|$ are the length for reduced sequences. Moreover, since \begin{align}\{K({\bf w}),S({\bf w}),T({\bf w)}\}=\{1,2,3\},\end{align} we have $[i]\mathcal{M}=\mathcal{T}^{\geq [i]}$ for any $i=1,2,3$. In particular, for any ${\bf w} \in \mathcal{T}^{\geq [i]}$, there exists a unique $X \in \mathcal{M}$ such that ${\bf w}=[i]X$.
\par
As in (\ref{eq: partial order on T}), we introduce a partial order $\leq$ on $\mathcal{M}$ as
\begin{equation}
X \leq Y \Longleftrightarrow \textup{$X$ is a prefix of $Y$.}
\end{equation}
When we focus on the tropical signs, it is important to separate the set $\mathcal{T}$ as follows.
\begin{definition}\label{definition: tree of sequences}
Fix the subtree $\mathcal{T}^{\geq [i]}$ on the direction $i=1,2,3$. We write $\mathcal{T}^{< [i]S^{\infty}}$ as the set of all $[i]S^{n} \in \mathcal{T}^{\geq [i]}$ with $n \in \mathbb{Z}_{\geq 0}$, and we call it the {\em trunk} of $\mathcal{T}^{\geq [i]}$. We say that a subset $\mathcal{T}^{\geq [i]X}$ ($X \in \mathcal{M}$) is a {\em branch} if the letter $T$ appears in the word $X$. Moreover, a branch of the form $\mathcal{T}^{\geq [i]S^nT}$ ($n \in \mathbb{Z}_{\geq 0}$) is called the {\em $n$-th maximal branch} of $\mathcal{T}^{\geq [i]}$.
\end{definition}
For any $i=1,2,3$, the set of the trunk and all maximal branches $\{\mathcal{T}^{< [i]S^{\infty}}\}\cup\{\mathcal{T}^{\geq [i]S^nT} \mid n \in \mathbb{Z}_{\geq 0}\}$ is a partition of $\mathcal{T}^{\geq [i]}$. Moreover, the set of $\mathcal{T}^{\geq [1]},\mathcal{T}^{\geq [2]},\mathcal{T}^{\geq [3]}$ and $\{\emptyset\}$ is a partition of $\mcT$. In fact, some conditions of tropical signs depend on whether ${\bf w}$ is in a trunk or a branch.
\par
The following notation is useful to study a rank $3$ cluster-cyclic exchange matrix.
\begin{definition}\label{def: K,S,T labbeling}
Fix a cluster-cyclic exchange matrix $B \in \mathrm{M}_3(\mathbb{R})$. For each ${\bf w} \in \mathcal{T}\backslash\{\emptyset\}$, define $K({\bf w})$, $S({\bf w})$, and $T({\bf w})$ as $k$, $s$, and $t$ defined in Theorem~\ref{thm: recursion for tropical signs}~($b$), respectively. For any $M,M'=K,S,T$, we write $\varepsilon_{M}^{\bf w}=\varepsilon_{M({\bf w})}^{\bf w}$, $b_{MM'}^{\bf w}=b_{M({\bf w}),M'({\bf w})}^{\bf w}$, and so on.
\end{definition}
The index $t=T({\bf w})$ can be distinguished from $s=S({\bf w})$ by $\varepsilon_{t}^{\bf w} = \varepsilon_{k}^{\bf w}$ or $\varepsilon_{t}^{\bf w}\mathrm{sign}(b_{kt}^{\bf w})=1$, see Theorem~\ref{thm: recursion for tropical signs}~$(a)$. Now, we may simplify this condition as follows.
\begin{lemma}\label{lemma: recursion for K,S,T}
Let $B$ be a cluster-cyclic exchange matrix and ${\bf w} \in \mathcal{T}\backslash\{\emptyset\}$. Then, the following statements hold.
\\
\textup{$(a)$} We have
\begin{equation}\label{eq: identification of trunks and branches}
\begin{aligned}
\textup{${\bf w}$ is in a trunk}\ &\Longleftrightarrow\ \varepsilon_{T}^{\bf w} \neq \varepsilon_{K}^{\bf w} \neq \varepsilon_{S}^{\bf w} \ \Longleftrightarrow\ \varepsilon_{T}^{\bf w} \mathrm{sign}(b_{KT}^{\bf w})=1,\\
\textup{${\bf w}$ is in a branch}\ &\Longleftrightarrow\ \varepsilon_{T}^{\bf w} = \varepsilon_{K}^{\bf w} \neq \varepsilon_{S}^{\bf w} \ \Longleftrightarrow\ \varepsilon_{T}^{\bf w}\mathrm{sign}(b_{KT}^{\bf w})=-1.
\end{aligned}
\end{equation}
\textup{$(b)$}\ We have
\begin{equation}\label{eq: recursion of S mutation for K,S,T}
(K({\bf w}S),S({\bf w}S),T({\bf w}S))=(S({\bf w}),K({\bf w}),T({\bf w})),
\end{equation}
and
\begin{equation}\label{eq: recursion of T mutation for K,S,T}
(K({\bf w}T),S({\bf w}T),T({\bf w}T))=\begin{cases}
(T({\bf w}),S({\bf w}),K({\bf w})), & \textup{if ${\bf w}$ is in a trunk},\\
(T({\bf w}),K({\bf w}),S({\bf w})), & \textup{if ${\bf w}$ is in a branch}.
\end{cases}
\end{equation}
\end{lemma}
\begin{proof}
For each ${\bf w}$, define the statement $(a)_{\bf w}$ and $(b)_{\bf w}$ as $(a)$ and $(b)$ for this ${\bf w}$. We show that the following three statements by induction.
\begin{itemize}
\item $(a)_{[i]}$ for any $i=1,2,3$.
\item $(a)_{{\bf w}} \Rightarrow (b)_{{\bf w}}$ for any ${\bf w} \in \mathcal{T}\backslash\{\emptyset\}$.
\item $(a)_{{\bf w}},(b)_{{\bf w}} \Rightarrow (a)_{{\bf w}S},(a)_{{\bf w}T}$ for any ${\bf w} \in \mathcal{T}\backslash\{\emptyset\}$.
\end{itemize}
\par
The first one is shown by (\ref{eq: initial conditions of tropical signs}). Note that $[i]$ is in a trunk.
\par
Next, we show $(a)_{\bf w} \Rightarrow (b)_{\bf w}$. By Theorem~\ref{thm: recursion for tropical signs}, we have
\begin{equation}\label{eq: next tropical signs}
\begin{array}{ll}
(\varepsilon_{K({\bf w})}^{{\bf w}S},\varepsilon_{S({\bf w})}^{{\bf w}S},\varepsilon_{T({\bf w})}^{{\bf w}S})=(-\varepsilon_{K}^{\bf w},-\varepsilon_{S}^{\bf w},\varepsilon_{T}^{\bf w}),\\
(\varepsilon_{K({\bf w})}^{{\bf w}T},\varepsilon_{S({\bf w})}^{{\bf w}T},\varepsilon_{T({\bf w})}^{{\bf w}T})=(\varepsilon_{K}^{\bf w},\varepsilon_{S}^{\bf w},-\varepsilon_{T}^{\bf w}).
\end{array}
\end{equation}
Firstly, we show (\ref{eq: recursion of S mutation for K,S,T}). Note that $K({\bf w}S)=S({\bf w})$ holds by definition. If we show $S({\bf w}S)=K({\bf w})$, then $T({\bf w}S)=T(\mathbf{w})$ is automatically determined. Thus, it suffices to show $S({\mathbf{w}S})=K(\mathbf{w})$. The first condition $\varepsilon_{K}^{{\bf w}S} \neq \varepsilon_{K({\bf w})}^{{\bf w}S}$ is shown by (\ref{eq: next tropical signs}). We also have
\begin{equation}
\varepsilon_{K({\bf w})}^{{\bf w}S}\mathrm{sign}(b_{K({\bf w}S),K({\bf w})}^{{\bf w}S})\overset{(\ref{eq: next tropical signs})}{=}-\varepsilon_{K}^{\bf w}\mathrm{sign}(b_{S({\bf w}),K({\bf w})}^{{\bf w}S}).
\end{equation}
Since $B$ is cluster-cyclic, we have $\mathrm{sign}(b_{S({\bf w}),K({\bf w})}^{{\bf w}S})=-\mathrm{sign}(b_{S({\bf w}),K({\bf w})}^{{\bf w}})=\mathrm{sign}(b_{K({\bf w}),S({\bf w})}^{\bf w})$. Moreover, by Theorem~\ref{thm: recursion for tropical signs}~$(a)$, we have $\varepsilon_{K}^{\bf w}=-\varepsilon_{S}^{\bf w}$. Thus, we have
\begin{equation}
\varepsilon_{K({\bf w})}^{{\bf w}S}\mathrm{sign}(b_{K({\bf w}S),K({\bf w})}^{{\bf w}S})=\varepsilon_{S}^{\bf w}\mathrm{sign}(b_{KS}^{\bf w})=-1,
\end{equation}
where the last equality can be obtained by Theorem~\ref{thm: recursion for tropical signs}~$(a)$. Thus, we have $S({\bf w}S)=K({\bf w})$.
Next, we show (\ref{eq: recursion of T mutation for K,S,T}). By definition, we obtain that $K(\mfw T)=T(\mfw)$. When ${\bf w}$ is in a trunk, we may show $\varepsilon_{K({\bf w})}^{{\bf w}T}=\varepsilon_{K}^{{\bf w}T}$ as follows:
\begin{equation}
\varepsilon_{K({\bf w})}^{{\bf w}T}\overset{(\ref{eq: next tropical signs})}{=}\varepsilon_{K}^{\bf w}\overset{(a)_{\bf w}}{=}-\varepsilon_{T}^{\bf w}\overset{(\ref{eq: next tropical signs})}{=}\varepsilon_{T({\bf w})}^{{\bf w}T}=\varepsilon_{K}^{{\bf w}T}.
\end{equation}
This equality implies that $T({\bf w}T)=K({\bf w})$ by \Cref{thm: recursion for tropical signs} $(a)$. When ${\bf w}$ is in a branch, we may show $\varepsilon_{S({\bf w})}^{{\bf w}T}=\varepsilon_{K}^{{\bf w}T}$ as follows:
\begin{equation}
\varepsilon_{S({\bf w})}^{{\bf w}T}\overset{(\ref{eq: next tropical signs})}{=}\varepsilon_{S}^{\bf w}=-\varepsilon_{K}^{\bf w}\overset{(a)_{\bf w}}{=}-\varepsilon_{T}^{\bf w}\overset{(\ref{eq: next tropical signs})}{=}\varepsilon_{T({\bf w})}^{{\bf w}T}=\varepsilon_{K}^{{\bf w}T},
\end{equation}
where the second equality is obtained by Theorem~\ref{thm: recursion for tropical signs}~$(a)$. Moreover, this implies that $T({\bf w}T)=S({\bf w})$.
\par
Lastly, we prove $(a)_{{\bf w}},(b)_{{\bf w}} \Rightarrow (a)_{{\bf w}S},(a)_{{\bf w}T}$. By $(b)_{\bf w}$, (\ref{eq: next tropical signs}), and $\varepsilon_{S}^{\bf w}=-\varepsilon_{K}^{\bf w}$, we have
\begin{equation}
\begin{aligned}
(\varepsilon_{K}^{{\bf w}S},\varepsilon_{T}^{{\bf w}S})&=(-\varepsilon_{S}^{\bf w},\varepsilon_{T}^{\bf w})=(\varepsilon_{K}^{\bf w},\varepsilon_{T}^{\bf w}),
\\
(\varepsilon_{K}^{{\bf w}T},\varepsilon_{T}^{{\bf w}T})&=\begin{cases}
(-\varepsilon_{T}^{\bf w},\varepsilon_{K}^{\bf w}), & \textup{if ${\bf w}$ is in a trunk},\\
(-\varepsilon_{T}^{\bf w},\varepsilon_{S}^{\bf w}) = (-\varepsilon_{T}^{\bf w},-\varepsilon_{K}^{\bf w}), & \textup{if ${\bf w}$ is in a branch}.
\end{cases}
\end{aligned}
\end{equation}
Thus, we may check the first equivalent condition. Moreover, since $B$ is cluster-cyclic, we have $\mathrm{sign}(b_{KT}^{{\bf w}S})=\mathrm{sign}(b_{S({\bf w}),T({\bf w})}^{{\bf w}S})=-\mathrm{sign}(b_{ST}^{\bf w})=\mathrm{sign}(b_{KT}^{\bf w})$, and
\begin{equation}
\varepsilon_{T}^{{\bf w}S}\mathrm{sign}(b_{KT}^{{\bf w}S})=\varepsilon_{T}^{\bf w}\mathrm{sign}(b_{KT}^{\bf w}).
\end{equation}
Hence, we proved the claim $(a)_{{\bf w}S}$. For $(a)_{{\bf w}T}$, by a similar argument with
\begin{equation}
\mathrm{sign}(b_{KT}^{{\bf w}T})=\begin{cases}
\mathrm{sign}(b_{KT}^{\bf w}), & \textup{if ${\bf w}$ is in a trunk},\\
-\mathrm{sign}(b_{KT}^{\bf w}), & \textup{if ${\bf w}$ is in a branch},
\end{cases}
\end{equation}
we may show that $\varepsilon_{T}^{{\bf w}T}\mathrm{sign}(b_{KT}^{{\bf w}T})=-1$. Here, note that ${\bf w}T$ is in a branch. This completes the proof.
\end{proof}

\subsection{Tropical signs in a trunk}\

Firstly, we give the explicit tropical signs indexed by $K,S,T$.
\begin{lemma}\label{lem: tropical signs in trunks as K,S,T}
Let ${\bf w}=[i]S^n \in \mathcal{T}^{< [i]S^{\infty}}$ be in a trunk. Then, we have
\begin{equation}\label{eq: tropical signs in trunks as K,S,T labeling}
\varepsilon_{K}^{\bf w}=-1,\quad \varepsilon_{S}^{\bf w}=\varepsilon_{T}^{\bf w}=1.
\end{equation}
\end{lemma}
\begin{proof}
If $n=0$, we may show it by (\ref{eq: initial conditions of tropical signs}). Suppose that the claim holds for some ${\bf w}=[i]S^n$. Then, we may show
\begin{equation}
\varepsilon_{K}^{{\bf w}S}=\varepsilon_{S({\bf w})}^{{\bf w}S}=-\varepsilon_{S}^{\bf w}=-1.
\end{equation}
Note that the first equality may be obtained by $K({\bf w}S)=S({\bf w})$ (Lemma~\ref{lemma: recursion for K,S,T}~$(b)$). The other case $\varepsilon_{S}^{\bf w}=\varepsilon_{T}^{\bf w}=1$ is similar.
\end{proof}
Then, we can obtain the expression of tropical signs for trunks explicitly.
\begin{proposition} \label{tropical trunk}
Fix any $i\in \{1,2,3\}$ and set $k_0=K([i])(=i)$, $s_0=S([i])$, and $t_0=T([i])$. For any $n\in \mbN$, the following statements hold.
\\
\textup{$(a)$} We have
\begin{equation}
(K([i]S^n),S([i]S^n))=\begin{cases}
(k_0,s_0), & \textup{if $n$ is even},\\
(s_0,k_0), & \textup{if $n$ is odd},
\end{cases}
\quad
T([i]S^n)=t_0.
\end{equation}
\\
\textup{$(b)$} We have
\begin{equation}
\varepsilon_{k_0}^{[i]S^n}=(-1)^{n+1},\ \varepsilon_{s_0}^{[i]S^n}=(-1)^n,\ \varepsilon_{t_0}^{[i]S^n}=1.
\end{equation}
\end{proposition}
\begin{proof}
We may check $(a)$ by Lemma~\ref{lemma: recursion for K,S,T}. Furthermore,  since $(a)$ holds, $(b)$ is immediately proved by Lemma~\ref{lem: tropical signs in trunks as K,S,T}.
\end{proof}

\subsection{Tropical signs in a branch}\

By considering Lemma~\ref{lemma: recursion for K,S,T}, properties of signs depend on whether ${\bf w}$ is in a trunk or a branch. Here, we focus on properties in a branch. Fix a subtree $\mathcal{T}^{\geq [i]}$ on the direction $i=1,2,3$. Let $\mathcal{T}^{\geq [i]X_0}$ be a branch of $\mathcal{T}^{\geq [i]}$, that is, we suppose that $X_0 \in \mathcal{M}$ contains at least one letter $T$.
\par
If we consider tropical signs indexed by $K$, $S$, $T$, we may give a simple formula.
\begin{proposition} \label{tropical branch}
If ${\bf w} \in \mathcal{T}\backslash\{\emptyset\}$ is in a branch, then for any $M=K,S,T$, we have
\begin{equation}\label{eq: recursion for tropical signs as K,S,T labeling}
\varepsilon_{M}^{{\bf w}S}=\varepsilon_{M}^{\bf w},\quad
\varepsilon_{M}^{{\bf w}T}=-\varepsilon_{M}^{\bf w}.
\end{equation}
In particular, for any $X \in \mathcal{M}$, we have
\begin{equation}\label{eq: tropical signs in branches as K,S,T labeling}
\varepsilon_{K}^{[i]S^nTX}=\varepsilon_{T}^{[i]S^nTX}=-(-1)^{\#_{T}(X)},\quad \varepsilon_{S}^{[i]S^nTX}=(-1)^{\#_{T}(X)},
\end{equation}
where $\#_{T}(X)$ is the number of times the letter $T$ appears in $X$.
\end{proposition}
\begin{proof}
By (\ref{eq: recursion for tropical signs}) and Lemma~\ref{lemma: recursion for K,S,T} $(b)$, we may directly check (\ref{eq: recursion for tropical signs as K,S,T labeling}). For example, we may show $\varepsilon_{K}^{{\bf w}T}=-\varepsilon_{K}^{\bf w}$ by $K({\bf w}T)=T({\bf w})$, that is 
\begin{equation}
\varepsilon_{K}^{{\bf w}T}=\varepsilon_{T({\bf w})}^{{\bf w}T}\overset{(\ref{eq: recursion for tropical signs})}{=}-\varepsilon_{T}^{\bf w}\overset{(\ref{eq: identification of trunks and branches})}{=}-\varepsilon_{K}^{\bf w}.
\end{equation}
By (\ref{eq: recursion for tropical signs}) and (\ref{eq: tropical signs in trunks as K,S,T labeling}), we may show
\begin{equation}
\begin{aligned}
\varepsilon_{K}^{[i]S^nT}&\overset{(\ref{eq: recursion of T mutation for K,S,T})}{=}\varepsilon_{T([i]S^n)}^{([i]S^n)T}\overset{(\ref{eq: recursion for tropical signs})}{=}-\varepsilon_{T}^{[i]S^n}\overset{(\ref{eq: tropical signs in trunks as K,S,T labeling})}{=}-1,\\
\varepsilon_{S}^{[i]S^nT}&\overset{(\ref{eq: recursion of T mutation for K,S,T})}{=}\varepsilon_{S([i]S^n)}^{[i]S^nT}\overset{(\ref{eq: recursion for tropical signs})}{=}\varepsilon_{S}^{[i]S^n}\overset{(\ref{eq: tropical signs in trunks as K,S,T labeling})}{=}1,\\
\varepsilon_{T}^{[i]S^nT}&\overset{(\ref{eq: recursion of T mutation for K,S,T})}{=}\varepsilon_{K([i]S^n)}^{[i]S^nT}\overset{(\ref{eq: recursion for tropical signs})}{=}\varepsilon_{K}^{[i]S^n}\overset{(\ref{eq: tropical signs in trunks as K,S,T labeling})}{=}-1.
\end{aligned}
\end{equation}
Hence, by using (\ref{eq: recursion for tropical signs as K,S,T labeling}) we may show (\ref{eq: tropical signs in branches as K,S,T labeling}) by induction on $|X|$.
\end{proof}
Next, we consider tropical signs indexed by $1,2,3$. We start with the following observation.

\begin{lemma}\label{lemma: well definedness of the action for E}
Fix a branch $\mathcal{T}^{\geq {\bf w}_0}$ for some ${\bf w}_0$. For any ${\bf w} \in \mathcal{T}^{\geq {\bf w}_{0}}$, the following statements hold.
\\
$(a)$\ The number of $j \in \{1,2,3\}\backslash\{k\}$ satisfying $\varepsilon_j^{\bf w}=\varepsilon_k^{\bf w}$ (and $\varepsilon_j^{\bf w} \neq \varepsilon_k^{\bf w}$) is precisely one.
In particular, two indices $S({\bf w})$ and $T({\bf w})$ can be recovered by the information $(\varepsilon_1^{\bf w},\varepsilon_2^{\bf w},\varepsilon_3^{\bf w};K({\bf w}))$.
\\
$(b)$\ For any $M=S,T$, the tuple $(\varepsilon_1^{{\bf w}M},\varepsilon_2^{{\bf w}M},\varepsilon_3^{{\bf w}M};K({\bf w}M))$ can be determined by the previous one $(\varepsilon_1^{\bf w},\varepsilon_2^{\bf w},\varepsilon_3^{\bf w};K({\bf w}))$.
\end{lemma}
\begin{proof}
The claim ($a$) follows from the fact that $S({\bf w})$ and $T({\bf w})$ exist and Lemma~\ref{lemma: recursion for K,S,T}. Namely, we choose $S({\bf w})$ as the index $j \in \{1,2,3\}\backslash\{k\}$ such that $\varepsilon_{j}^{\bf w} \neq \varepsilon_{k}^{\bf w}$. The claim ($b$) follows from Theorem~\ref{thm: recursion for tropical signs} and ($a$). (Thanks to the equality (\ref{eq: recursion for tropical signs}), we can determine $\varepsilon_j^{{\bf w}M}$. Note that $S({\bf w})$ and $T({\bf w})$ are determined by $(\varepsilon_1^{\bf w},\varepsilon_2^{\bf w},\varepsilon_3^{\bf w};K({\bf w}))$.)
\end{proof}
Thus, tropical signs are determined by $(\varepsilon_1^{\bf w},\varepsilon_2^{\bf w},\varepsilon_3^{\bf w};K({\bf w}))$ if we focus on a branch. (Namely, we do not have to consider $B$-matrices and the mutation sequences ${\bf w}$.) For any ${\bf w} \in \mathcal{T}^{\geq {\bf w}_0}$, set
\begin{equation}
E^{\bf w}=(\varepsilon_1^{\bf w},\varepsilon_2^{\bf w},\varepsilon_3^{\bf w};K({\bf w})).
\end{equation}

\begin{definition}\label{def: collection of signs}
For each branch $\mathcal{T}^{\geq {\bf w}_0}$, let $\mathcal{E}^{\geq {\bf w}_0}$ be the collection of all $E^{\bf w}$ indexed by ${\bf w} \in \mathcal{T}^{\geq {\bf w}_0}$. We define maps $K,S,T:\mathcal{E}^{\geq {\bf w}_0}\to\{1,2,3\}$ and a right monoid action of $\mathcal{M}$ on $\mathcal{E}^{\geq {\bf w}_0}$ as follows. For each $E^{\bf w} \in \mathcal{E}^{\geq {\bf w}_0}$, we define maps $M(E^{\bf w})=M({\bf w})$ for any $M=K,S,T$ and for any $X \in \mathcal{M}$, we define a right monoid action $(E^{\bf w})^{X}=E^{{\bf w}X}$ on $\mathcal{E}^{\geq {\bf w}_0}$.
\end{definition}
\begin{lemma}\label{lem: set E}
For any branch $\mathcal{T}^{\geq {\bf w}_0}$, the set of all $E^{\bf w}$ with ${\bf w} \in \mathcal{T}^{\geq {\bf w}_0}$ is given by
\begin{equation}\label{eq: set E}
\begin{aligned}
&\ \{(\varepsilon_1,\varepsilon_2,\varepsilon_3;k) \mid\ \textup{The number of $j \in \{1,2,3\}$ with $\varepsilon_j^{\bf w} \neq \varepsilon_k^{\bf w}$ is $1$}\}\\
=&\ \left\{\begin{aligned}
(+,+,-;1),(+,-,+;1),(-,-,+;1),(-,+,-;1)\\
(-,+,+;2),(+,+,-;2),(+,-,-;2),(-,-,+;2)\\
(+,-,+;3),(-,+,+;3),(-,+,-;3),(+,-,-;3)
\end{aligned}\right\}.
\end{aligned}
\end{equation}
In particular, this set does not depend on the choice of a branch $\mathcal{T}^{\geq {\bf w}_0}$.
\end{lemma}
\begin{proof}
By Lemma~\ref{lemma: well definedness of the action for E}, we may show that every $E^{\bf w}$ in a branch should belong to the set in (\ref{eq: set E}). Moreover, by a direct calculation, we may find all the elements as exhibited in Figure~\ref{fig: proof of the diheadral group}. Therefore, the claim holds.

\begin{figure}[hpbt]
\begin{tikzpicture}
\draw (2,0) node {$(+,+,-;1)$};
\draw (2,-1) node {$(+,-,-;2)$};
\draw (2,-2) node {$(+,-,+;3)$};
\draw (2,-3) node {$(-,-,+;1)$};
\draw (2,-4) node {$(-,+,+;2)$};
\draw (2,-5) node {$(-,+,-;3)$};
\draw (-2,0) node {$(-,+,+;3)$};
\draw (-2,-1) node {$(-,+,-;1)$};
\draw (-2,-2) node {$(+,+,-;2)$};
\draw (-2,-3) node {$(+,-,-;3)$};
\draw (-2,-4) node {$(+,-,+;1)$};
\draw (-2,-5) node {$(-,-,+;2)$};
\foreach \x in{0,1,2,3,4}
    {
    \draw[<-] (-2,-0.3-\x) -- (-2,-0.5-\x) node [right] {$T$} -> (-2,-0.7-\x);
    };
\foreach \x in{0,1,2,3,4}
    {
    \draw[->] (2,-0.3-\x) -- (2,-0.5-\x) node [right] {$T$} -> (2,-0.7-\x);
    };
\foreach \x in{0,1,2,3,4,5}
    {
    \draw[<->] (-1,-\x)-- (0,-\x) node [above] {$S$}--(1,-\x);
    };
\draw[->] (3,-5) to [out=30, in=-30] (3,0);
\draw[->] (-3,0) to [out=-150, in=150] (-3,-5);
\draw (-4.5,-2.5) node {$T$};
\draw (4.5,-2.5) node {$T$};
\end{tikzpicture}
\caption{Mutations of tropical signs in a branch}\label{fig: proof of the diheadral group}
\end{figure}
\end{proof}
Let $\mathcal{E}$ be the set as (\ref{eq: set E}).
Thanks to Lemma~\ref{lemma: well definedness of the action for E}, the maps $K,S,T$ and the action defined in Definition~\ref{def: collection of signs} can be respectively seen as maps and an action on $\mathcal{E}$. That is, they are determined by the tuple $(\varepsilon_1,\varepsilon_2,\varepsilon_3;k)$ but not ${\bf w}$. Based on this fact, we introduce the maps $K,S,T:\mathcal{E} \to \{1,2,3\}$ and the action on $\mathcal{E}$. In particular, the maps $K,S,T: \mathcal{E} \to \{1,2,3\}$ can be expressed as
\begin{equation}\label{eq: map S,T in a branch}
\begin{aligned}
K(\varepsilon_1,\varepsilon_2,\varepsilon_3;k)&=k,\\
S(\varepsilon_1,\varepsilon_2,\varepsilon_3;k)&=[\textup{the index}\ j \neq k\ \textup{such that}\ \varepsilon_j \neq \varepsilon_k],\\
T(\varepsilon_1,\varepsilon_2,\varepsilon_3;k)&=[\textup{the index}\ j \neq k\ \textup{such that}\ \varepsilon_j = \varepsilon_k].
\end{aligned}
\end{equation}
Moreover, since the action of $\mathcal{M}$ on $\mathcal{T}\setminus\{\emptyset\}$ is a right monoid action, we can prove that this action on $\mathcal{E}$ is also a right monoid action.
\par
We introduce an equivalence relation $\sim$ on $\mathcal{M}$ as
\begin{equation}
X \sim Y \Longleftrightarrow E^{X}=E^{Y}\quad (\forall E \in \mathcal{E}). \label{eq: equivalent relation of M}
\end{equation}
This relation is compatible with the monoid structure of $\mathcal{M}$ (namely, if $X \sim Y$ and $X' \sim Y'$, then $XX' \sim YY'$ holds), so we may consider its quotient monoid $\overline{\mathcal{M}}=\mathcal{M}/{\sim}$. For any $X \in \mathcal{M}$, we write its equivalence class by $\bar{X}$. By convention, we write $\mathrm{id}=\overline{1_{\mathcal{M}}}$, where $1_{\mathcal{M}}$ is the identity element of the monoid $\mathcal{M}$.
\par
By the definition of $\sim$, a faithful action of $\overline{\mathcal{M}}$ on $\mathcal{E}$ is induced as follows:
\begin{equation}
E^{\bar{X}}=E^{X} \quad(\forall E \in \mathcal{E}, X \in \mathcal{M}).
\end{equation}
The structure of $\overline{\mathcal{M}}$ is related to a well-known group structure.
\begin{theorem}\label{thm: group structure of M}
The quotient monoid $\overline{\mathcal{M}}$ is a group which is isomorphic to the dihedral group $\mathcal{D}_{6}$ of order $12$. In particular, the following relations hold:
\begin{equation}\label{eq: relations of M}
\bar{S}^2=\bar{T}^6=(\bar{S}\bar{T})^2=\mathrm{id}.
\end{equation}
\end{theorem}
\begin{proof}
This can be directly shown by the mutation relation in Figure~\ref{fig: proof of the diheadral group}.
\end{proof}

\subsection{Fractal structure in the $\varepsilon$-pattern}\

Theorem~\ref{thm: group structure of M} can be stated as a fractal phenomenon appearing in the $\varepsilon$-pattern. To state the claim, we introduce the following group action.
\begin{definition}
We introduce a right group action of $\mathfrak{S}_3$ on $\mathcal{E}$ as
\begin{equation}
(\varepsilon_1,\varepsilon_2,\varepsilon_3;k)^{\sigma} = (\varepsilon_{\sigma(1)},\varepsilon_{\sigma(2)}, \varepsilon_{\sigma(3)};\sigma^{-1}(k)).
\end{equation}
Set $\nu=\bar{T}^3$. Note that it acts on $\mathcal{E}$ as \begin{equation}
(\varepsilon_1,\varepsilon_2,\varepsilon_3;k)^{\nu}=(-\varepsilon_1,-\varepsilon_2,-\varepsilon_3;k).
\end{equation}
Since $\nu^2=\mathrm{id}$, we may consider the subgroup $\{\mathrm{id},\nu\}$ of $\overline{\mathcal{M}}$. Then, we write the group of products as $\mathfrak{S}_3^{\nu}=\{\mathrm{id},\nu\}\times\mathfrak{S}_3$.
\end{definition}
We can easily check that $(E^{\sigma})^{\nu}=(E^{\nu})^{\sigma}$ for any $\sigma \in \mathfrak{S}_3$. Thus, there is no confusion by identifying $(\mathrm{id},\sigma)=\sigma$ and writing $(\nu,\sigma)=\nu\sigma=\sigma\nu$.
\par
By a direct calculation, we have
\begin{equation}\label{eq: permutation and labeling}
M(E^{\sigma})=\sigma^{-1}(M(E)),
\quad
M(E^{\nu})=M(E)
\end{equation}
for any $E \in \mathcal{E}$, $\sigma \in \mathfrak{S}_3$ and $M=K,S,T$.
\begin{lemma}\label{lem: transitivity of permutations}
The group $\mathfrak{S}_{3}^{\nu}$ acts on $\mathcal{E}$ simply transitively. Namely, for any $E=(\varepsilon_1,\varepsilon_2,\varepsilon_3;k),E'=(\varepsilon_1',\varepsilon_2',\varepsilon_3';k') \in \mathcal{E}$, there exists a unique $\xi \in \mathfrak{S}_3^{\nu}$ such that $E'=E^{\xi}$. 
\end{lemma}
\begin{proof}
In fact, we can define $\xi\in \mathfrak{S}_3^{\nu}$ as follows:
\begin{itemize}
\item Let $\sigma \in \mathfrak{S}_3$ satisfy $\sigma(M(E'))=M(E)$ for any $M=K,S,T$.
\item If $\varepsilon_{k}=\varepsilon_{k'}'$, let $\xi = \sigma$. If $\varepsilon_{k} \neq \varepsilon_{k'}'$, let $\xi = \nu\sigma$.
\end{itemize}
By the definition of the action of $\sigma$, the $k'$th component of $E^{\sigma}$ is $\sigma^{-1}(k')=k$. This means that the $k'$th component of $E'$ and $E^{\xi}$ are the same. Let $s=S(E)$, $s'=S(E')$, $t=T(E)$, $t'=T(E')$. Since $\varepsilon_{k}=\varepsilon_t\neq\varepsilon_s$ and $\varepsilon_{k'}'=\varepsilon_{t'}' \neq \varepsilon_{s'}^{'}$, we may do the same argument for the $s'$th and the $t'$th components. Thus, we may show the existence of $\xi$. Moreover, since $\#(\mathcal{E})=\#(\mathfrak{S}_3^{\nu})=12$, this $\xi$ should be unique.
\end{proof}
\begin{theorem}[Fractal structure in $\varepsilon$-pattern]\label{thm: fractal structure in E-pattern}
The following statements hold.
\\
\textup{($a$)} The actions of $\mathfrak{S}_3^{\nu}$ and $\overline{\mathcal{M}}$ on $\mathcal{E}$ are always commutative.
\\
\textup{($b$)} Let ${\bf w}_0,{\bf w}_1$ be in a branch. Let $\xi \in \mathfrak{S}_3^{\nu}$ satisfy
\begin{equation}\label{eq: definition of xi in fractal structure theorem}
E^{{\bf w}_1}=(E^{{\bf w}_0})^{\xi}.
\end{equation}
(By Lemma~\ref{lem: transitivity of permutations}, such $\xi$ uniquely exists.)
Then, this $\xi$ induces a one-to-one correspondence between $\mathcal{E}^{\geq {\bf w}_0}$ and $\mathcal{E}^{\geq {\bf w}_1}$ by
\begin{equation}\label{eq: fractal structure theorem}
E^{{\bf w}_1X}=(E^{{\bf w}_0X})^{\xi}
\quad
(\forall X \in \mathcal{M}).
\end{equation}
Moreover, we express $\xi=\lambda\sigma$ for some $\lambda \in \{\mathrm{id},\nu\}$ and $\sigma \in \mathfrak{S}_3$. Then, for any ${\bf u} \in \mathcal{T}\setminus\mathcal{T}^{\geq [K({\bf w}_1)]}$, we have
\begin{equation}\label{eq: fractal relation}
E^{{\bf w}_1{\bf u}}=(E^{{\bf w}_0(\sigma{\bf u})})^{\xi},
\end{equation}
where $\sigma\mathbf{u}=[\sigma(k_1),\sigma(k_2),\dots,\sigma(k_r)]$ for $\mathbf{u} = [k_1,\dots,k_r]$.
\end{theorem}
\begin{proof}
($a$) Since all the actions are compatible with their group structures, it suffices to show the equality $(E^{\xi})^{\bar{X}}=(E^{\bar{X}})^{\xi}$ ($\forall E \in \mathcal{E}$) for all generators $\xi \in \mathfrak{S}_3^{\nu}$ and $\bar{X} \in \overline{\mathcal{M}}$. In particular, it suffices to show the case of $(\xi,\bar{X})=(\nu,\bar{S})$, $(\nu,\bar{T})$, $(\sigma,\bar{S})$, and $(\sigma,\bar{T})$, where $\sigma \in \mathfrak{S}_3$. When $\xi=\nu$, we can check it by a direct calculation. Consider the case of $\xi=\sigma \in \mathfrak{S}_3$ and $\bar{X}=\bar{S}$. Set $E=(\varepsilon_1,\varepsilon_2,\varepsilon_3;k)$ and $s=S(E)$, $t=T(E)$.
By (\ref{eq: recursion of S mutation for K,S,T}) and (\ref{eq: permutation and labeling}), we have
\begin{equation}
\begin{aligned}
(K((E^{\sigma})^{\bar{S}}),S((E^{\sigma})^{\bar{S}}),T((E^{\sigma})^{\bar{S}}))&\overset{(\ref{eq: recursion of S mutation for K,S,T})}{=}(K((E^{\bar{S}})^{\sigma}),S((E^{\bar{S}})^{\sigma}),T((E^{\bar{S}})^{\sigma}))\\
&\overset{(\ref{eq: permutation and labeling})}{=}(\sigma^{-1}(s),\sigma^{-1}(k),\sigma^{-1}(t)).
\end{aligned}
\end{equation}
Moreover the $\sigma^{-1}(k)$th components of both $(E^{\sigma})^{\bar{S}}$ and $(E^{\bar{S}})^{\sigma}$ are the same as $-\varepsilon_{k}$. By the second condition of (\ref{eq: identification of trunks and branches}), we can easily check that all components are the same. Thus, the equality $(E^{\sigma})^{\bar{S}}=(E^{\bar{S}})^{\sigma}$ holds. We may show $(E^{\sigma})^{\bar{T}}=(E^{\bar{T}})^{\sigma}$ by a similar argument.
\\
($b$) By the commutativity of ($a$), we can show (\ref{eq: fractal structure theorem}) as follows:
\begin{equation}
(E^{{\bf w}_1})^{X}\overset{(\ref{eq: definition of xi in fractal structure theorem})}{=}((E^{{\bf w}_0})^{\xi})^{X}\overset{(a)}{=}((E^{{\bf w}_0})^{X})^{\xi}=(E^{{\bf w}_0X})^{\xi}.
\end{equation}
Next, we show (\ref{eq: fractal relation}). Since ${\bf u} \in \mathcal{T}\setminus\mathcal{T}^{\geq [K({\bf w}_1)]}$, we may express ${\bf w}_1{\bf u}={\bf w}_1X$ for some $X \in \mathcal{M}$. By (\ref{eq: fractal structure theorem}), we have $E^{{\bf w}_1{\bf u}}=(E^{{\bf w}_0X})^{\xi}$. Thus, it suffices to show $E^{{\bf w}_0X}=E^{{\bf w}_0(\sigma{\bf u})}$. Set $X=M_1M_2\cdots M_r$ and ${\bf u}=[k_1,k_2,\dots,k_r]$. Since $E^{{\bf w}_1{\bf u}}=E^{{\bf w}_1X}$ holds, they satisfy $k_{i}=M_{i}(E^{{\bf w}_1[k_1,k_2,\dots,k_{i-1}]})$ for any $i=1,2,\dots,r$. By Lemma~\ref{lem: transitivity of permutations}, $\sigma$ is determined by $\sigma(M(E^{{\bf w}_1}))=M(E^{{\bf w}_0})$ for any $M=K,S,T$. In particular, we have $M_1(E^{{\bf w}_0})=\sigma(M_1(E^{{\bf w}_1}))=\sigma(k_1)$.  Moreover, we have $k_2=M_2(E^{{\bf w}_1[k_1]})=M_2(E^{{\bf w}_1M_1})=M_2((E^{{\bf w}_0M_1})^{\xi})=\sigma^{-1}(M_2(E^{{\bf w}_0M_1}))=\sigma^{-1}(M_2(E^{{\bf w}_0[\sigma(k_1)]}))$, where the third equality follows from ($a$) and the fourth equality follows from (\ref{eq: permutation and labeling}). This implies $M_2(E^{{\bf w}_0[\sigma(k_1)]})=\sigma(k_2)$. By repeating this argument, we obtain $M_{i}(E^{{\bf w}_0(\sigma[k_1,k_2,\dots,k_{i-1}])})=\sigma(k_i)$, and it implies that $E^{{\bf w}_0(\sigma{\bf u})}=E^{{\bf w}_0X}$ as we desired.
\end{proof}

\begin{example}
Take one ${\bf w}_0$ such that $E^{{\bf w}_0}=(+,+,-;1)$. Then, by obeying the rules (\ref{eq: map S,T in a branch}) and (\ref{eq: recursion for tropical signs}), we may calculate $\mathcal{E}^{\geq {\bf w}_0}$ as in Figure~\ref{fig: E0}. Let ${\bf w}_1={\bf w}_0T$. Then, we can calculate $\mathcal{E}^{\geq {\bf w}_1} \subset \mathcal{E}^{\geq {\bf w}_0}$ as in Figure~\ref{fig: E1}. Here, we put the $S$-mutated tuple at the lower left, and put the $T$-mutated tuple at the lower right. Then, these two collections $\mathcal{E}^{\geq {\bf w}_1} \subset \mathcal{E}^{\geq {\bf w}_0}$ can be connected by an action $\xi=\nu\sigma$, where $\sigma=(1,3,2) \in \mathfrak{S}_3$ is a cyclic permutation. Note that, by (\ref{eq: fractal relation}), each edge labeled by $i=1,2,3$ in $\mathcal{E}^{\geq {\bf w}_0}$ is sent to an edge in $\mathcal{E}^{\geq {\bf w}_1}$ labeled by $\sigma^{-1}(i)$.
\begin{figure}[hbtp]
\centering
\begin{minipage}{0.49\linewidth}
\centering
\begin{tikzpicture}
\draw (0,0) node [draw, rectangle] {$(+,+,-;1)$};
\draw (1.5,0) node {${\bf w}_0$};
\draw (-2,-1.5) node {$(-,+,+;3)$};
\draw (2,-1.5) node [draw, rectangle] {$(+,-,-;2)$};
\draw (3.5,-1.5) node {${\bf w}_1$};
\draw (-3,-3) node {$(+,+,-;1)$};
\draw (-1,-3) node {$(-,-,+;2)$};
\draw (1,-3) node {$(-,+,-;1)$};
\draw (3,-3) node {$(+,-,+;3)$};
\draw[->] (-0.5,-0.25)--(-1,-0.75) node [above left] {$S=3$}->(-1.5,-1.25);
\draw[->] (0.5,-0.25)--(1,-0.75) node [above right] {$T=2$}->(1.5,-1.25);
\draw[->] (-2.25,-1.75)--(-2.5,-2.25) node [left] {$S=1$}->(-2.75,-2.75);
\draw[->] (-1.25,-1.75)->(-0.75,-2.75);
\draw (-0.9,-2.2) node [above right] {$T=2$};
\draw[->] (1.25,-1.75)->(0.75,-2.75);
\draw (0.9,-2.1) node [below left] {$S=1$};
\draw[->] (2.25,-1.75)--(2.5,-2.25) node [right] {$T=3$}->(2.75,-2.75);
\end{tikzpicture}
\caption{$\mathcal{E}^{\geq {\bf w}_0}$}\label{fig: E0}
\end{minipage}
\begin{minipage}{0.49\linewidth}
\centering
\begin{tikzpicture}
\draw (0,0) node [draw, rectangle] {$(+,-,-;2)$};
\draw (1.5,0) node {${\bf w}_1$};
\draw (-2,-1.5) node {$(-,+,-;1)$};
\draw (2,-1.5) node {$(+,-,+;3)$};
\draw (-3,-3) node {$(+,-,+;2)$};
\draw (-1,-3) node {$(-,+,+;3)$};
\draw (1,-3) node {$(+,+,-;2)$};
\draw (3,-3) node {$(-,-,+;1)$};
\draw[->] (-0.5,-0.25)--(-1,-0.75) node [above left] {$S=1$}->(-1.5,-1.25);
\draw[->] (0.5,-0.25)--(1,-0.75) node [above right] {$T=3$}->(1.5,-1.25);
\draw[->] (-2.25,-1.75)--(-2.5,-2.25) node [left] {$S=2$}->(-2.75,-2.75);
\draw[->] (-1.25,-1.75)->(-0.75,-2.75);
\draw (-0.9,-2.2) node [above right] {$T=3$};
\draw[->] (1.25,-1.75)->(0.75,-2.75);
\draw (0.9,-2.1) node [below left] {$S=2$};
\draw[->] (2.25,-1.75)--(2.5,-2.25) node [right] {$T=1$}->(2.75,-2.75);
\end{tikzpicture}
\caption{$\mathcal{E}^{\geq {\bf w}_1}$}\label{fig: E1}
\end{minipage}
\end{figure}
\end{example}

\section{A geometric characterization via polygons} \label{geometric characterization}
In this section, we aim to give a geometric model of tropical signs associated with rank $3$ real cluster-cyclic exchange matrices, which will be beneficial to the calculation. Moreover, we equip the dihedral group $\mcD_6$ with a cluster realization via $S,T$-mutations in a branch, that is the quotient monoid $\overline{\mathcal{M}}$.
\subsection{A geometric model of the tropical signs}\

 Firstly, we label all the $12$ tropical signs as follows:
 \begin{equation}
 	\begin{array}{cc}
 		A_1=(+,+,-;1),\ B_1=(+,-,-;2),\ C_1=(+,-,+;3),\\
 		D_1=(-,-,+;1),\ E_1=(-,+,+;2),\ F_1=(-,+,-;3),\\
 		A_2=(-,+,+;3),\ B_2=(-,+,-;1),\ C_2=(+,+,-;2),\\
 		D_2=(+,-,-;3),\ E_2=(+,-,+;1),\ F_2=(-,-,+;2).
 	\end{array} 
 \end{equation} Then, we have $B_1=\bar{T}(A_1),\ C_1=\bar{T}(B_1),\ D_1=\bar{T}(C_1),\ E_1=\bar{T}(D_1),\ F_1=\bar{T}(E_1),\ A_1=\bar{T}(F_1).$ Similarly, we also have the relation from $A_2$ to $F_2$. 
 
 In addition, we can check that $\bar{S}(A_1)=A_2,\ \bar{S}(B_1)=B_2,\ \bar{S}(C_1)=C_2,\ \bar{S}(D_1)=D_2,\ \bar{S}(E_1)=E_2,\ \bar{S}(F_1)=F_2.$ Note that $\bar{S}$ is an involution. Given a dodecagon, we mark its $6$ vertices evenly by $\{A_1,B_1,C_1,D_1,E_1,F_1\}$, such that each pair of adjacent vertices differs by a $60$-degree rotation $\bar{T}$. Then, we choose an axis of reflection corresponding to $\bar{S}$ that does not pass through any vertex arbitrarily. We can refer to \Cref{geo model graph} as an example. Then, according to the reflection of $\bar{S}$, we mark the left $6$ vertices evenly by $\{A_2,B_2,C_2,D_2,E_2,F_2\}$ such that the pairs $\{A_1,A_2\}$, $\{B_1,B_2\}$, $\{C_1,C_2\}$, $\{D_1,D_2\}$, $\{E_1,E_2\}$ and $\{F_1,F_2\}$ are symmetric with respect to the axis of reflection. It can be checked that all the vertices in the dodecagon are compatible with the $60$-degree counterclockwise rotation $\bar{T}$ and the reflection $\bar{S}$.
 
 By constructing such a geometric model of tropical signs, we can directly read the group actions of $\overline{\mcM}$ over the tropical signs in the geometric model instead of complicated calculations. 
 \begin{example}
 	Suppose $E^{\mfw}=A_1=(+,+,-;1)$ for some reduced sequence $\mfw$ and let $X=\bar{T}^3\bar{S}\in \overline{\mcM}$. Then, we can directly obtain that $E^{\mfw X}=D_2=(+,-,-;3)$ from \Cref{geo model graph}.
\begin{figure}[htbp]
\centering
\begin{tikzpicture}[scale=1.5, dot/.style={circle, fill, inner sep=1.5pt}, >={Classical TikZ Rightarrow[length=8, width=4]}]

\clip (-5,-2.5) rectangle (6,2.5);

\node[dot] (F_1) at ({2*cos(15)},{2*sin(15)}) {};
\node[dot] (E_2) at ({2*cos(45)},{2*sin(45)}) {};
\node[dot] (A_1) at ({2*cos(75)},{2*sin(75)}) {};
\node[dot] (D_2) at ({2*cos(105)},{2*sin(105)}) {};
\node[dot] (B_1) at ({2*cos(135)},{2*sin(135)}) {};
\node[dot] (C_2) at ({2*cos(165)},{2*sin(165)}) {};
\node[dot] (C_1) at ({2*cos(195)},{2*sin(195)}) {};
\node[dot] (B_2) at ({2*cos(225)},{2*sin(225)}) {};
\node[dot] (D_1) at ({2*cos(255)},{2*sin(255)}) {};
\node[dot] (A_2) at ({2*cos(285)},{2*sin(285)}) {};
\node[dot] (E_1) at ({2*cos(315)},{2*sin(315)}) {};
\node[dot] (F_2) at ({2*cos(345)},{2*sin(345)}) {};

\foreach \x in {A_1,B_1,C_1,D_1,E_1,F_1}
    {
    \node at ($1.2*(\x)$) {$\x$};
    }
\foreach \x in {A_2,B_2,C_2,D_2,E_2,F_2}
    {
    \node[red] at ($1.2*(\x)$) {$\x$};
    }
\draw (F_1)--(E_2)--(A_1)--(D_2)--(B_1)--(C_2)--(C_1)--(B_2)--(D_1)--(A_2)--(E_1)--(F_2)--(F_1);
\draw[blue, dashed] (-3,0)--(3,0) node (S) [right] {axis of reflection};

\draw[blue, ->] (A_1) -- (B_1) node [pos=0.3, auto] {$\bar{T}$};
\draw[blue, ->] (B_1) -- (C_1) node [pos=0.3, auto] {$\bar{T}$};
\draw[blue, ->] (C_1) -- (D_1) node [pos=0.3, auto] {$\bar{T}$};
\draw[blue, ->] (D_1) -- (D_2) node [pos=0.3, auto=right] {$\bar{S}$};

\end{tikzpicture}
\caption{Geometric model of tropical signs via the dodecagon}
\label{geo model graph}
\end{figure}
 \end{example}

 \subsection{A cluster realization of the dihedral group $\mathcal{D}_{6}$}\
 
 Now, we aim to give a cluster realization of the dihedral group $\mathcal{D}_{6}$ based on the geometric model as above. Without loss of generality, we focus on \Cref{geo model graph} as an example. We glue the pair of adjacent vertices together as follows:
 \begin{align}
 	\{A_1,D_2\},\ \{B_1,C_2\},\ \{C_1,B_2\},\ \{D_1,A_2\},\ \{E_1,F_2\},\ \{F_1,E_2\}. 
 \end{align} That is to say, we regard the two different vertices in each pair as the same one. Then, the dodecagon reduces to the $6$-gon, see \Cref{cluster real}. We can directly check that $\bar{T}$ corresponds to the $60$-degree counterclockwise rotation and $\bar{S}$ corresponds to the reflection with respect to $S$. Note that $\bar{T}^6=\bar{S}^2=(\bar{S}\bar{T})^2=\id$, which implies that it is a presentation of the dihedral group $\mcD_6$. Since $\bar{T}$ and $\bar{S}$ arise from cluster mutations associated with the tropical signs, this gives a cluster realization of $\mathcal{D}_6$.

\begin{figure}[htbp]
\centering
\begin{tikzpicture}[scale=2]
\begin{scope}[rotate=30]

\foreach \i in {1,...,6} {
  \coordinate (P\i) at ({cos(60*(\i-1))}, {sin(60*(\i-1))});
}

\foreach \i/\label in {1/$F_1(\textcolor{red}{E_2})$, 2/$A_1(\textcolor{red}{D_2})$, 3/$B_1(\textcolor{red}{C_2})$, 4/$C_1(\textcolor{red}{B_2})$, 5/$D_1(\textcolor{red}{A_2})$, 6/$E_1(\textcolor{red}{F_2})$} {
  \fill (P\i) circle (1.5pt); 
  \node at ($(0,0)!1.4!(P\i)$) {\label};
}

\foreach \i in {1,...,6} {
  \pgfmathtruncatemacro{\j}{mod(\i,6)+1}
  \draw[thick] (P\i) -- (P\j);
}

\draw[blue, dashed, thick] 
  ({sqrt(3)/2*1.6}, {-0.5*1.6}) -- ({-sqrt(3)/2*1.6}, {0.5*1.6});

\node[blue] at ({sqrt(3)/2*1.8 }, {-0.5*1.8}) {$\overline{S}$};

\draw[-{Latex}, blue, thick, dashed]
  ([shift={(0.15,0.1)}]-5:0.5) arc[start angle=0, end angle=74, radius=0.45];
  \node at ([shift={(-0.1,0.18)}]0:0.45) {$\textcolor{blue}{\overline{T}}$};
 \end{scope}
\end{tikzpicture}
\caption{Cluster realization of the dihedral group $\mathcal{D}_{6}$}
\label{cluster real}
\end{figure}

\subsection*{Acknowledgements} 
 The authors would like to sincerely thank Tomoki Nakanishi for his thoughtful guidance. The authors are also grateful to Peigen Cao, Yasuaki Gyoda, Kyungyong Lee, Salvatore Stella and Toshiya Yurikusa for their valuable discussions and insightful suggestions. In addition, Z. Chen  wants to thank Xiaowu Chen, Zhe Sun and Yu Ye for their help and support. R. Akagi is supported by JSPS KAKENHI Grant Number JP25KJ1438 and Chubei Itoh Foundation. Z. Chen is supported by National Natural Science Foundation of China (Grant No. 124B2003) and China Scholarship Council (Grant No. 202406340022).
 


\newpage
\begin{thebibliography}{99}
\newcommand{\au}[1]{\textrm{#1},}
\newcommand{\ti}[1]{\textrm{#1},}
\newcommand{\jo}[1]{\textit{#1}}
\newcommand{\vo}[1]{\textbf{#1}}
\newcommand{\yr}[1]{(#1)}
\newcommand{\pp}[2]{#1--#2.}
\newcommand{\arxiv}[1]{\href{http://arxiv.org/abs/#1}{arXiv:#1}}
\bibitem[Aka24]{Aka24}
\au{R. Akagi}
\ti{Cluster-cyclic condition of skew-symmetrizable matrices of rank 3 via Markov constant}
 \yr{2024}, arXiv: 2411.07083.
\bibitem[AC25]{AC25}
\au{R. Akagi, Z. Chen} \ti{Real $C$-, $G$-structures and sign-coherence of cluster algebras} \yr{2025},
arXiv: 2509.06486.
\bibitem[BBH11]{BBH11}
\au{A. Beineke, T. Br\"ustle and L. Hille} 
\ti{Cluster-cyclic quivers with three vertices and the Markov equation}
\jo{Algebr. Represent. Theory}
{\bf 14} \yr{2011}, no.~1, \pp{97}{112} MR2763295
\bibitem[BGZ06]{BGZ06}
\au{M. Barot, C. Geiss and A. Zelevinsky}
\ti{Cluster algebras of finite type and positive symmetrizable
matrices}
\jo{J. London Math. Soc. }
\vo{73}, \yr{2006}, no. 3, \pp{545}{564} MR2241966 
\bibitem[BHIT17]{BHIT17}
\au{T. Br{\"u}stle, S. Hermes, K. Igusa, G. Todorov}
\ti{Semi-invariant pictures and two conjectures on maximal green sequences}
\jo{J. Algebra}
\vo{473}, \yr{2017}, \pp{80}{109} MR3591142
\bibitem[Cha15]{Cha15}
\au{A. N\'ajera Ch\'{a}vez}
\ti{On the $c$-Vectors of an Acyclic Cluster Algebra}
\jo{Int. Math. Res. Not.}
\yr{2015}, no. 6, \pp{1590}{1600} MR3340366 
\bibitem[CGY22]{CGY22}
\au{P. Cao, Y. Gyoda, T. Yurikusa}
\ti{Bongartz Completion via $c$-Vectors}
\jo{Int. Math. Res. Not.}
\yr{2023}, no. 15, \pp{13099}{13135} MR4621860
\bibitem[CL24]{CL24}
\au{Z.~Chen, Z.~Li}
\ti{Mutation invariants of cluster algebras of rank 2}
\jo{J. Algebra}
\vo{654} \yr{2024}, \pp{25}{58} MR4749386 
\bibitem[CL25]{CL25}
\au{Z.~Chen, Z.~Li}
\ti{Sign-equivalence in cluster algebras: Classification and applications to Markov-type equations}
\jo{J. Pure Appl. Algebra}
\vo{229} \yr{2025}, 108058. MR4941197 
\bibitem[DWZ10]{DWZ10}
\au{H. Derksen, J. Weyman, A. Zelevinsky} 
\ti{Quivers with potentials and their representations II: applications to cluster algebras}
\jo{J. Amer. Math. Soc.} \vo{23(3)} \yr{2010}, \pp{749}{790} MR2629987
\bibitem[EJLN24]{EJLN24}
\au{T. J. Ervin, B. Jackson, K. Lee, S. D. Nguyen} 
\ti{Geometry of $c$-vectors and $c$-matrices for mutation-infinite quivers} \yr{2024},
arXiv:2410.08510.
\bibitem[FT18]{FT18}
\au{A. Felikson, P. Tumarkin}
\ti{Acyclic cluster algebras, reflection groups, and curves on a punctured disc}
\jo{Adv. Math.}
\vo{340} \yr{2018}, \pp{855}{882} MR3886182
\bibitem[FZ02]{FZ02}
\au{S. Fomin, A. Zelevinsky}
\ti{Cluster algebras I: Foundations}
\jo{J. Amer. Math. Soc.}
\vo{15} \yr{2002}, \pp{497}{529} MR1887642 
\bibitem[FZ03]{FZ03}
\au{S. Fomin, A. Zelevinsky}
\ti{Cluster algebras II. Finite type classification}
\jo{Invent. Math}
\vo{154}, \yr{2003}, no. 1, \pp{63}{121} MR2004457 
\bibitem[FZ07]{FZ07}
\au{S. Fomin, A. Zelevinsky}
\ti{Cluster algebras IV: Coefficients}
\jo{Comp. Math.}
\vo{143} \yr{2007}, \pp{112}{164} MR2295199 
\bibitem[GHKK18]{GHKK18}
\au{M. Gross, P. Hacking, S. Keel, M. Kontsevich}
\ti{Canonical bases for cluster algebras}
\jo{J. Amer. Math. Soc.}
\vo{31}, \yr{2018}, \pp{497}{608} MR3758151 
\bibitem[GM23]{GM23}
\au{Y. Gyoda, K. Matsushita}
\ti{Generalization of Markov Diophantine equation via generalized cluster algebra}
\jo{Electron. J. Comb.}
\vo{30}(4), \yr{2023}, P4.10. MR4657283

\bibitem[HK16]{HK16}
\au{A. Hubery, H. Krause} 
\ti{A categorification of non-crossing partitions} \jo{J. Eur. Math. Soc.} \vo{18} \yr{2016}, no. 10, \pp{2273}{2313} MR3551191 
\bibitem[Lam16]{Lam16}
\au{P. Lampe}
\ti{Diophantine equations via cluster transformations}
\jo{J. Algebra}
\vo{462} \yr{2016}, \pp{320}{337} MR3519509 
\bibitem[LL24]{LL24}
\au{J. Lee, K. Lee} 
\ti{An unexpected property of $\mathbf{g}$-vectors for rank 3 mutation-cyclic quivers} \yr{2024},
arXiv:2409.00599.
\bibitem[LLM23]{LLM23}
\au{K. Lee, K. Lee, M. Mills} 
\ti{Geometric description of $c$-vectors and real L{\"o}sungen}
\jo{Math. Z.} 
\vo{303},
\yr{2023}, no. 2, 44. MR4537352
\bibitem[LMN24]{LMN24}
\au{F. Lin, G. Musiker, T. Nakanishi}
\ti{Two formulas for $F$-Polynomials}
\jo{Int. Math. Res. Not.}
\yr{2024}, no. 1, \pp{613}{634} MR4686662 
\bibitem[Nag13]{Nag13}
\au{K. Nagao}
\ti{Donaldson-Thomas theory and cluster algebras}
\jo{Duke Math. J.}
\vo{162}(7), \yr{2013}, \pp{1313}{1367} MR3079250 
\bibitem[Nak21]{Nak21}
\au{T. Nakanishi}
\ti{Synchronicity phenomenon in cluster patterns}
\jo{J. London Math. Soc.}
\vo{103}, \yr{2021}, \pp{1120}{1152} MR4245832
\bibitem[Nak23]{Nak23}
\au{T. Nakanishi}
\ti{Cluster algebras and scattering diagrams}
\jo{MSJ Mem.} \vo{41} \yr{2023}, 279 pp; ISBN: 978-4-86497-105-8. MR4563311 
\bibitem[Ngu22]{Ngu22}
\au{S. Nguyen} 
\ti{A proof of Lee-Lee's conjecture about geometry of rigid modules} 
\jo{J. Algebra} \vo{611} \yr{2022}, \pp{422}{434} MR4474642
\bibitem[NZ12]{NZ12}
\au{T. Nakanishi, A. Zelevinsky}
\ti{On tropical dualities in cluster algebras}
\jo{Contemp. Math.} \vo{565} \yr{2012}, \pp{217}{226}  MR2932428 
\bibitem[Pla11]{Pla11}
\au{P. Plamondon} 
\ti{Cluster algebras via cluster categories with infinite-dimensional morphism spaces}
\jo{Compos. Math.} \vo{147} \yr{2011}, \pp{1921}{1954} MR2862067 
\bibitem[Rea14]{Rea14}
\au{N. Reading}
\ti{Universal geometric cluster algebras}
\jo{Math. Z.} \vo{277} \yr{2014} \pp{499}{547} MR3205782 
\bibitem[RS16]{RS16}
\au{N. Reading and D.~E. Speyer}
\ti{Combinatorial frameworks for cluster algebras}
\jo{Int. Math. Res. Not. IMRN}
\yr{2016} no.~1, \pp{109}{173} MR3514060
\bibitem[Sev11]{Sev11}
\au{A. Seven}
\ti{Cluster algebras and semipositive symmetrizable matrices} 
\jo{Trans. Amer. Math.
Soc.} \vo{363} \yr{2011}, \pp{2733}{2762} MR2763735 
\bibitem[Sev12]{Sev12}
\au{A. Seven}
\ti{Mutation classes of $3\times 3$ generalized Cartan matrices}
\jo{Highlights in Lie Algebraic Methods, Progr. Math.} \vo{295} \yr{2012}, \pp{205}{211} MR2866853 
\bibitem[Sev14]{Sev14}
\au{A. Seven}
\ti{Maximal green sequences of skew-symmetrizable $3\times 3$ matrices}
\jo{Linear Algebra Appl.} \vo{440} \yr{2014}, \pp{125}{130} MR3134258 
\bibitem[Sev15]{Sev15}
\au{A. Seven}
\ti{Cluster algebras and symmetric matrices}
\jo{Proc. Amer. Math.
Soc.} \vo{143}(2) \yr{2015}, \pp{469}{478} MR3973884 
\bibitem[ST13]{ST13}
\au{D. Speyer, H. Thomas} 
\ti{Acyclic Cluster Algebras Revisited}
\jo{Algebras, Quivers and Representations} \yr{2013}, \pp{275}{298} MR3183889 
\bibitem[War14]{War14}
\au{M. Warkentin} 
\ti{Exchange graphs via quiver mutation} Dissertation \yr{2014}, 103.
\end{thebibliography}
\end{document}